\documentclass[envcountsame,reqno]{llncs}
\usepackage{times,amsmath,amssymb,enumerate,gastex}

\newcommand{\M}{\mathbb{M}}
\newcommand{\N}{\mathbb{N}}
\newcommand{\dA}{\mathbb{A}}
\newcommand{\dB}{\mathbb{B}}
\newcommand{\dC}{\mathbb{C}}
\newcommand{\dD}{\mathbb{D}}
\newcommand{\dG}{\mathbb{G}}
\newcommand{\dM}{\mathbb{M}}
\newcommand{\dP}{\mathbb{P}}
\newcommand{\dS}{\mathbb{S}}
\newcommand{\dT}{\mathbb{T}}

\newcommand{\dV}{\mathbb{V}}

\newcommand{\alp}{\mathsf{alph}}
\newcommand{\Aut}{\mathsf{Aut}}

\newcommand{\CR}{\mathsf{core}}

\newcommand{\Inn}{\mathsf{Inn}}
\newcommand{\IRR}{\mathsf{IRR}}
\newcommand{\NF}{\mathsf{NF}}
\newcommand{\Out}{\mathsf{Out}}

\newcommand{\RSCCP}{\mathsf{RSCCP}}
\newcommand{\SCCP}{\mathsf{CCP}}
\newcommand{\val}{\mathsf{val}}

\newcommand{\amalgam}{\otimes}

\renewcommand{\phi}{\varphi}
\renewcommand{\epsilon}{\varepsilon}

\begin{document}

\title{Compressed Conjugacy and the Word Problem for Outer
  Automorphism Groups of Graph Groups}

\author{Niko Haubold \and Markus Lohrey \and Christian Mathissen}

\institute{Institut f\"ur Informatik, Universit\"at Leipzig, Germany\\
  \texttt{\{haubold,lohrey,mathissen\}@informatik.uni-leipzig.de} }

\maketitle

\begin{abstract}
  It is shown that for graph groups (right-angled Artin groups)
  the conjugacy problem
  as well as a restricted version of the simultaneous
  conjugacy problem can be solved
  in polynomial time even if input words are represented
  in a compressed form. As a consequence it follows
  that the word problem for the outer automorphism group of a
  graph group  can be solved in polynomial time.
\end{abstract}

\section{Introduction}

{\em Automorphism groups} and {\em outer automorphism groups} of {\em
  graph groups} received a lot of interest in the past few
years. A graph group $\dG(\Sigma,I)$ is given by a finite
undirected graph $(\Sigma,I)$ (without self-loops).  The set $\Sigma$
is the set of generators of $\dG(\Sigma,I)$ and every edge $(a,b) \in
I$ gives rise to a commutation relation $ab=ba$.  Graph groups are
also known as {\em right-angled Artin groups} or {\em free partially
  commutative groups}.  Graph groups interpolate between finitely
generated free groups and finitely generated free Abelian groups.  The
automorphism group of the free Abelian group $\mathbb{Z}^n$ is
$\mathsf{GL}(n,\mathbb{Z})$ and hence finitely generated. By a
classical result of Nielsen, also automorphism groups of free groups
are finitely generated, see e.g. \cite{LySch77}. For graph groups in
general, it was shown by Laurence \cite{Lau95} (building up on
previous work by Servatius \cite{Ser89}) that their automorphism
groups are finitely generated. Only recently, Day \cite{Day09} has
shown that $\Aut(\dG(\Sigma,I))$ is always finitely presented.  Some
recent structural results on automorphism groups of graph groups can
be found in \cite{ChCrVo07,ChoVog09}; for a survey see \cite{Cha07}.

In this paper, we continue the investigation of algorithmic aspects of
automorphism groups of graph groups. In \cite{LoSchl07} it was shown
that the word problem for $\Aut(\dG(\Sigma,I))$ can be solved in
polynomial time. The proof of this result used compression techniques.
It is well-known that the word problem for $\dG(\Sigma,I)$ can be
solved in linear time. In \cite{LoSchl07}, a compressed (or succinct)
version of the word problem for graph groups was studied. In this
variant of the word problem, the input word is represented succinctly
by a so-called \emph{straight-line program}. This is a context free
grammar $\dA$ that generates exactly one word $\val(\dA)$, see
Section~\ref{S SLP}.  Since the length of this word may grow
exponentially with the size (number of productions) of the SLP $\dA$,
SLPs can be seen indeed as a succinct string representation. SLPs
turned out to be a very flexible compressed representation of strings,
which are well suited for studying algorithms for compressed data, see
e.g. \cite{GaKaPlRy96,Lif07,Loh06siam,MiShTa97,Pla94,PlaRy98}.  In
\cite{LoSchl07,Schl06} it was shown that the word problem for the
automorphism group $\Aut(G)$ of a group $G$ can be reduced in
polynomial time to the {\em compressed word problem} for $G$, where
the input word is succinctly given by an SLP.  In \cite{Schl06}, it
was shown that the compressed word problem for a finitely generated
free group $F$ can be solved in polynomial time and in \cite{LoSchl07}
this result was extended to graph groups.  It follows that the word
problem for $\Aut(\dG(\Sigma,I))$ can be solved in polynomial time.
Recently, Macdonald \cite{Macd09} has shown that also the compressed
word problem for every fully residually free group can be solved in
polynomial time.

It is not straightforward to carry over the above mentioned complexity
results from $\Aut(\dG(\Sigma,I))$ to the {\em outer} automorphism
group $$\Out(\dG(\Sigma,I)) =
\Aut(\dG(\Sigma,I))/\text{Inn}(\dG(\Sigma,I)).$$ Nevertheless,
Schleimer proved in \cite{Schl06} that the word problem for the outer
automorphism group of a finitely generated free group can be decided
in polynomial time.  For this, he used a compressed variant of the
simultaneous conjugacy problem in free groups. In this paper, we
generalize Schleimer's result to graph groups: For every graph
$(\Sigma,I)$, the word problem for $\Out(\dG(\Sigma,I))$ can be solved
in polynomial time.  Analogously to Schleimer's approach for free
groups, we reduce the word problem for $\Out(\dG(\Sigma,I))$ to a
compressed variant of the simultaneous conjugacy problem in
$\dG(\Sigma,I)$. In this problem, we have given an SLP $\dA_a$ for
every generator $a \in \Sigma$, and the question is whether there
exists $x \in \dG(\Sigma,I)$ such that $a = x\, \val(\dA_a)\, x^{-1}$ for
all $a \in \Sigma$. A large part of this paper develops a polynomial
time algorithm for this problem.  Moreover, we also present a
polynomial time algorithm for the compressed version of the classical
conjugacy problem in graph groups: In this problem, we have given two
SLPs $\dA$ and $\dB$ and we ask whether there exists $x \in
\dG(\Sigma,I)$ such that $\val(\dA) = x\, \val(\dB) x^{-1}$ in
$\dG(\Sigma,I)$.  For our polynomial time algorithm, we have
to develop a pattern matching algorithm for SLP-compressed
Mazurkiewicz traces, which is inspired by a pattern matching
algorithm for hierarchical message sequence charts from \cite{GenMus08}.
For the non-compressed version of the 
conjugacy problem in $\dG(\Sigma,I)$, a
linear time algorithm was presented in \cite{Wra89} based on
\cite{LiuWraZeg90}.  In \cite{CrGoWi09} this result was generalized to
various subgroups of graph groups.

\section{Preliminaries}

Let $\Sigma$ be a finite alphabet.  For a word $s = a_1 \cdots a_m$
($a_i \in \Sigma$) let
\begin{itemize}
\item $|s| = m$, $\alp(s) = \{a_1, \ldots, a_m\}$,
\item $s[i] = a_i$ for $1 \leq i \leq m$,
\item $s[i:j] = a_i \cdots a_j$ for $1 \leq i \leq j \leq m$ and
  $s[i:j] = \varepsilon$ for $i > j$, and
\item $|s|_a = |\{k \mid  s[k]=a\}|$ for $a \in \Sigma$.
\end{itemize}
We use $\Sigma^{-1} = \{ a^{-1} \mid  a \in \Sigma\}$ to denote a
disjoint copy of $\Sigma$ and let $\Sigma^{\pm 1} = \Sigma \cup
\Sigma^{-1}$. Define $(a^{-1})^{-1} = a$; this defines an involution
${}^{-1} : \Sigma^{\pm 1} \to \Sigma^{\pm 1}$, which can be extended
to an involution on $(\Sigma^{\pm 1})^*$ by setting $(a_1 \cdots
a_n)^{-1} = a_n^{-1} \cdots a_1^{-1}$.

\subsection{Straight-line programs}
\label{S SLP}

We are using straight-line programs as a succinct representation of
strings with reoccurring subpatterns \cite{PlRy99}. A
\textit{straight-line program (SLP) over the alphabet $\Gamma$} is a
context free grammar $\dA=(V,\Gamma,S,P)$, where $V$ is the set of
\textit{nonterminals}, $\Gamma$ is the set of \textit{terminals},
$S\in V$ is the \textit{initial nonterminal}, and $P\subseteq V \times
(V\cup \Gamma)^*$ is the set of \textit{productions} such that (i) for
every $X\in V$ there is exactly one $\alpha \in (V \cup \Gamma)^*$
with $(X,\alpha)\in P$ and (ii) there is no cycle in the relation
$\{(X,Y)\in V\times V \mid  \exists \alpha : (X,\alpha) \in P, Y 
\in \alp(\alpha) \}$. These conditions ensure that the language
generated by the straight-line program $\dA$ contains exactly one word
$\val(\dA)$. Moreover, every nonterminal $X\in V$
generates exactly one word that is denoted by $\val_\dA(X)$, or
briefly $\val(X)$, if $\dA$ is clear from the context.  The size of
$\dA$ is $|\dA|= \sum_{(X,\alpha)\in P}|\alpha|$. It can be seen
easily that an SLP can be transformed in polynomial time into an
equivalent SLP in \textit{Chomsky normal form}, which means that all
productions have the form $A\rightarrow BC$ or $A\rightarrow a$ with
$A,B,C\in V$ and $a \in \Gamma$.

For an SLP $\dA$ over $\Sigma^{\pm 1}$ (w.l.o.g. in Chomsky normal
form) we denote with $\dA^{-1}$ the SLP that has for each terminal
rule $A\rightarrow a$ from $\dA$ the terminal rule $A\rightarrow
a^{-1}$ and for each nonterminal rule $A\rightarrow BC$ from $\dA$ the
nonterminal rule $A\rightarrow CB$.  Clearly, $\val(\dA^{-1}) =
\val(\dA)^{-1}$.

Let us state some simple algorithmic problems that
can be easily solved in polynomial time:
\begin{itemize}
\item Given an SLP $\dA$, calculate $|\val(\dA)|$.
\item Given an SLP $\dA$ and a number $i \in \{1, \ldots, |\val(\dA)|\}$, calculate
$\val(\dA)[i]$.
\item Given an SLP $\dA$ and two numbers $1 \leq i \leq j \leq  |\val(\dA)|$,
compute and SLP $\dB$ with $\val(\dB) = \val(\dA)[i,j]$.
\end{itemize}
In \cite{Pla94}, Plandowski presented a polynomial time algorithm
for testing whether $\val(\dA) = \val(\dB)$ for two given 
SLPs $\dA$ and $\dB$. A cubic algorithm was presented by Lifshits \cite{Lif07}.
In fact, Lifshits gave an algorithm for compressed pattern
matching: given SLPs $\dA$ and $\dB$, is $\dA$ a factor of $\dB$?
The running time of his algorithm is $O(|\dA| \cdot |\dB|^2)$.

A \emph{composition system} $\dA = (V,\Gamma,S,P)$ is defined
analogously to an SLP, but in addition to productions of the form
$A\to \alpha$ ($A\in V$, $\alpha\in (V\cup\Gamma)^*$) it may also
contain productions of the form $A \to B[i:j]$ for $B\in V$ and $i,j
\in \mathbb{N}$ \cite{GaKaPlRy96}. For such a production we define
$\val_{\dA}(A) = \val_{\dA}(B)[i:j]$.
The size of a production $A \to B[i:j]$ is
$\lceil\log(i)\rceil+\lceil\log(j)\rceil$.  As for SLPs we define
$\val(\dA) = \val_{\dA}(S)$.  In \cite{Hag00}, Hagenah presented a
polynomial time algorithm, which transforms a given composition system
$\dA$ into an SLP $\dB$ with $\val(\dA) = \val(\dB)$.
Composition systems were also heavily used in \cite{LoSchl07,Schl06}
in order to solve compressed word problems efficiently.

\subsection{Trace monoids and graph groups}

We introduce some notions from trace theory. For a thorough
introduction see \cite{DieRoz95}. An \emph{independence alphabet} is a
finite undirected graph $(\Sigma,I)$ without loops. Then $I\subseteq
\Sigma \times \Sigma$ is an irreflexive and symmetric relation. The
complementary graph $(\Sigma,D)$ with $D=
(\Sigma\times\Sigma)\setminus I$ is called a \emph{dependence
  alphabet}. The \emph{trace monoid} $\dM(\Sigma,I)$ is defined as the
quotient $\dM(\Sigma,I)=\Sigma^*/\{ab=ba\mid (a,b)\in I\}$ with
concatenation as operation and the empty word as the neutral
element. This monoid is cancellative and its elements are called
\emph{traces}. The trace represented by the word $s\in \Sigma^*$ is
denoted by $[s]_I$. As an ease of notation we denote with
$u\in\Sigma^*$ also the trace $[u]_I$. For $a\in \Sigma$ let
$I(a)=\{b\in\Sigma\mid (a,b)\in I\}$ be the letters that commute with $a$
and $D(a)=\{b\in \Sigma\mid  (a,b)\in D\}$ be the letters that are
dependent from $a$. For traces $u,v\in \dM(\Sigma,I)$ we denote with
$uIv$ the fact that $\alp(u)\times \alp(v)\subseteq I$. For
$\Gamma\subseteq \Sigma$ we say that $\Gamma$ is \emph{connected} if the
subgraph of $(\Sigma,D)$ induced by $\Gamma$ is connected. For a trace
$u$ we denote with $\max(u)$ the set of possible last letters of $u$,
i.e., $\max(u) = \{a\mid u=va\text{ for } a\in\Sigma, v\in
\dM(\Sigma, I)\}$. Analogously we define $\min(u)$ to be the set of
possible first letters i.e., $\min(u) = \{a\mid u=av\text{ for }
a\in\Sigma,v\in \dM(\Sigma, I)\}$.

A convenient representation for traces are \emph{dependence graphs},
which are node-labeled directed acyclic graphs. For a word
$w \in\Sigma^*$ the dependence graph $D_w$ has vertex set
$\{1,\dots,|w|\}$ where the node $i$ is labeled with $w[i]$. There is an
edge from vertex $i$ to $j$ if and only if $i<j$ and $(w[i],w[j])\in
D$. It is easy to see that for two words $w,w'\in \Sigma^*$ we have
$[w]_I = [w']_I$ if and only if $D_w$ and $D_{w'}$ are isomorphic. Hence, we
can speak of \emph{the} dependence graph of a trace.

For background in combinatorial group theory see \cite{LySch77}. The
\emph{free group} $F(\Sigma)$ generated by $\Sigma$ can be defined as
the quotient monoid
$$
F(\Sigma) = (\Sigma^{\pm 1})^*/\{ aa^{-1} = \varepsilon \mid  a \in
\Sigma^{\pm 1} \}.
$$
For an independence alphabet $(\Sigma, I)$ the {\em graph group}
$\mathbb{G}(\Sigma,I)$ is defined as the quotient group
$$
\mathbb{G}(\Sigma,I) = F(\Sigma)/\{ab = ba \mid  (a,b) \in I\}.
$$
{}From the independence alphabet $(\Sigma,I)$
we derive the independence alphabet 
$$(\Sigma^{\pm 1}, \{(a^{\epsilon_1},b^{\epsilon_2})\mid (a,b)\in
I,\; \epsilon_1,\epsilon_2\in\{-1,1\}\}).$$ 
Abusing notation, we denote the independence relation 
of this alphabet again with $I$. 

Note that $(a,b)\in I$ implies $a^{-1}b = ba^{-1}$ in
$\mathbb{G}(\Sigma,I)$. Thus, the graph group $\mathbb{G}(\Sigma,I)$
can be also defined as the quotient
$$
\mathbb{G}(\Sigma,I) = \mathbb{M}(\Sigma^{\pm 1},I)/\{ aa^{-1} =
\varepsilon \mid  a \in \Sigma^{\pm 1} \}.
$$
Graph groups are also known as right-angled Artin groups and free
partially commutative groups.

\subsection{(Outer) automorphism groups}

The \emph{automorphism group} $\Aut(G)$ of a group $G$ is the set of
all bijective homomorphisms from $G$ to itself with composition as
operation and the identity mapping as the neutral element. An
automorphism $\varphi$ is called \emph{inner} if there is a group
element $x\in G$ such that $\varphi(y)=xyx^{-1}$ for all $y\in G$. 
The set of all inner automorphisms
for a group $G$ forms the \emph{inner automorphism group} of $G$
denoted by $\Inn(G)$. This is easily seen to be a normal subgroup of
$\Aut(G)$ and the quotient group $\Out(G)=\Aut(G)/\Inn(G)$
is called the \emph{outer automorphism group} of $G$.

Assume that $\Aut(G)$ is finitely generated\footnote{In general, this
  won't be the case, even if $G$ is finitely generated.}  and let
$\Psi=\{\psi_1,\dots,\psi_k\}$ be a monoid generating set for
$\Aut(G)$. Then $\Psi$ also generates $\Out(G)$ where we identify
$\psi_i$ with the coset $\psi_i\,\Inn(G)$ for $i\in\{1,\dots,k\}$.
Then the \emph{word problem} for the outer automorphism group can be
viewed as the following decision problem:

\medskip

\noindent
INPUT: A word $w\in\Psi^*$.\\
QUESTION: Does $w=1$ in $\Out(G)$?

\medskip
\noindent
Since an automorphism belongs to the same coset (with respect to $\Inn(G)$) as
the identity if and only if it is inner, we can rephrase the word
problem for $\Out(G)$ as follows:

\medskip
\noindent
INPUT: A word $w\in\Psi^*$. \\
QUESTION: Does $w$ represent an element of $\Inn(G)$ in $\Aut(G)$?

\medskip
\noindent
Building on results from \cite{Ser89}, Laurence has shown in
\cite{Lau95} that automorphism groups of graph groups are finitely
generated. Recently, Day \cite{Day09} proved that automorphism groups
of graph groups are in fact finitely presented. Further results on
(outer) automorphism groups of graph groups can be found in
\cite{ChCrVo07,ChoVog09}. The main purpose of this paper is to give a
polynomial time algorithm for the word problem for
$\Out(\mathbb{G}(\Sigma,I))$.

\section{Main results}\label{sec:main-results}

In this section we will present the main results of this paper the
proof of which are subject to the rest of the paper.  In order to
solve the word problem for $\Out(\mathbb{G}(\Sigma,I))$ in polynomial
time, we have to deal with compressed conjugacy problems in
$\mathbb{G}(\Sigma,I)$.  Recall that two elements $g$ and $h$ of a
group $G$ are {\em conjugated} if and only if there exists $x \in G$
such that $g = x h x^{-1}$.  The classical {\em conjugacy problem} for
$G$ asks, whether two given elements of $G$ are conjugated.  We will
consider a compressed variant of this problem in $\mathbb{G}(\Sigma,I)$,
which we call the \emph{compressed conjugacy problem for
  $\mathbb{G}(\Sigma,I)$}, $\SCCP(\Sigma,I)$ for short:

\medskip

\noindent
INPUT: SLPs $\dA$ and $\dB$ over $\Sigma^{\pm 1}$.\\
QUESTION: Are $\val(\dA)$ and $\val(\dB)$ conjugated in
$\mathbb{G}(\Sigma,I)$?

\begin{theorem}\label{decidesccp}
  Let $(\Sigma,I)$ be a fixed independence alphabet. Then,
  $\SCCP(\Sigma,I)$ can be solved in polynomial time.
\end{theorem}
We will proof Theorem~\ref{decidesccp} in Section~\ref{sec:CC}.

In order to solve the word problem for $\Out(\mathbb{G}(\Sigma,I))$ in
polynomial time, Theorem~\ref{decidesccp} is not sufficient.  We need
an extension of $\SCCP(\Sigma,I)$ to several pairs of input SLPs. Let
us call this problem the {\em simultaneous compressed conjugacy
  problem} for $\mathbb{G}(\Sigma,I)$:

\medskip

\noindent
INPUT: SLPs $\dA_1,\dB_1,\dots,\dA_n,\dB_n$ over $\Sigma^{\pm 1}$.\\
QUESTION: Does there exist $x\in{(\Sigma^{\pm 1})}^*$ such that
$\val(\dA_i)=x\,\val(\dB_i)x^{-1}$ in $\mathbb{G}(\Sigma,I)$ for all
$i\in\{1,\dots,n\}$?  \medskip

\noindent
The simultaneous (non-compressed) conjugacy problem also appears in
connection with group-based cryptography \cite{MyShUs08}.
Unfortunately, we don't know, whether the simultaneous compressed
conjugacy problem can be solved in polynomial time.  But, in order to
deal with the word problem for $\Out(\mathbb{G}(\Sigma,I))$, a
restriction of this problem suffices, where the SLPs
$\dB_1,\dots,\dB_n$ from the simultaneous compressed conjugacy problem
are the letters from $\Sigma$. We call this problem the {\em
  restricted simultaneous compressed conjugacy problem}, briefly
$\RSCCP(\Sigma,I)$:

\medskip

\noindent
INPUT: SLPs $\dA_a$ ($a \in \Sigma$) over $\Sigma^{\pm 1}$.\\
QUESTION: Does there exist $x\in{(\Sigma^{\pm 1})}^*$ with
$\val(\dA_a)=xax^{-1}$ in $\mathbb{G}(\Sigma,I)$ for all $a\in\Sigma$?

\medskip

\noindent
An $x$ such that $\val(\dA_a)=xax^{-1}$ in $\mathbb{G}(\Sigma,I)$ for
all $a\in\Sigma$ is called a {\em solution} of the
$\RSCCP(\Sigma,I)$-instance. The following theorem will be shown in
Section~\ref{sec:RSSC}.

\begin{theorem}\label{decide_ccp}
  Let $(\Sigma,I)$ be a fixed independence alphabet.  Then,
  $\RSCCP(\Sigma,I)$ can be solved in polynomial time.  Moreover, in
  case a solution exists, one can compute an SLP for a solution in
  polynomial time.
\end{theorem}
Using Theorem~\ref{decide_ccp}, it is straightforward to decide the
word problem for $\Out(\mathbb{G}(\Sigma,I))$ in polynomial time.

\begin{theorem} \label{thm:outer}
  Let $(\Sigma,I)$ be a fixed independence alphabet. Then, the word
  problem for the group $\Out(\mathbb{G}(\Sigma,I))$ can be solved in
  polynomial time.
\end{theorem}

\begin{proof}
  Fix a finite monoid generating set $\Phi$ for
  $\Aut(\mathbb{G}(\Sigma,I))$.  Let $\varphi=\varphi_1\cdots
  \varphi_n$ with $\varphi_1,\ldots, \varphi_n \in \Phi$ be the
  input. By \cite{Schl06} we can compute in polynomial time SLPs
  $\dA_a$ ($a \in \Sigma$) over $\Sigma^{\pm 1}$ with
  $\val(\dA_a)=\varphi(a)$ in $\mathbb{G}(\Sigma,I)$ for all $a \in
  \Sigma$.  The automorphism $\varphi$ is inner if and only if there
  exists $x$ such that $\val(\dA_a)=x a x^{-1}$ in
  $\mathbb{G}(\Sigma,I)$ for all $a \in \Sigma$.  This can be decided
  in polynomial time by Theorem~\ref{decide_ccp}.  \qed
\end{proof}
It is important in Theorem~\ref{decidesccp}--\ref{thm:outer}
that we fix the independence alphabet $(\Sigma,I)$. It is open
whether these results also hold if $(\Sigma,I)$ is part of
the input.

\section{Simple facts for traces}

In this section, we state some simple facts on the prefix order of
trace monoids, which will be needed later.  A trace $u$ is said to be
a \emph{prefix} of a trace $w$ if there exists a trace $v$ such that
$uv = w$ and we denote this fact by $u\preceq w$.  The prefixes of a
trace $w$ correspond to the downward-closed node sets of the
dependence graph of $w$. Analogously a trace $v$ is a \emph{suffix} of
a trace $w$ if there is a trace $u$ such that $uv=w$.  For two traces
$u,v \in \mathbb{M}(\Sigma,I)$, the \emph{infimum} $u \sqcap v$ is the
largest trace $s$ with respect to $\preceq$ such that $s \preceq u$ and $s
\preceq v$; it always exists \cite{CoMeZi93}.  With $u \setminus v$ we
denote the unique trace $t$ such that $u = (u \sqcap v) t$; uniqueness
follows from the fact that $\mathbb{M}(\Sigma,I)$ is cancellative.
Note that $u \setminus v = u \setminus (u \sqcap v)$.  

The \emph{supremum} $u \sqcup v$ of two traces $u,v \in
\mathbb{M}(\Sigma,I)$ is the smallest trace $s$ with respect to $\preceq$ such
that $u\preceq s$ and $v\preceq s$ if any such trace exists. The
following result can be found in \cite{CoMeZi93}:

\begin{lemma}[\cite{CoMeZi93}] \label{lemma-supremum-exists} The trace
  $u\sqcup v$ exists if and only if $(u\setminus v)\, I\, ( v\setminus
  u)$, in which case we have $u\sqcup v=(u\sqcap v)\,(u\setminus v)\,
  ( v\setminus u)$.
\end{lemma}
We can define the supremum of several traces $w_1,\ldots, w_n$
analogously by induction: $w_1\sqcup\dots \sqcup w_n=(w_1\sqcup \dots
\sqcup w_{n-1})\sqcup w_n$. We mention a necessary and sufficient
condition for the existence of the supremum of several traces that
follows directly from the definition.

\begin{lemma}\label{trace_supremum}
  Let $(\Sigma,I)$ be an independence alphabet and
  $u_1,\dots,u_r\in\dM(\Sigma,I)$. If $ u=u_1\sqcup\dots\sqcup
  u_{r-1}$ exists then $s=u_1\sqcup\dots\sqcup u_r$ is exists if and
  only if $\left( u\setminus u_r \right)\;I\;\left( u_r\setminus
    u\right)$.  In this case $s= u\;\left( u_r\setminus u\right)$.
\end{lemma}

\begin{example} \label{Ex traces1} We consider the following
  independence alphabet $(\Sigma,I)$:
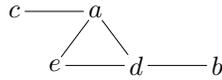
  \begin{center}
    \setlength{\unitlength}{1.2mm}
    \begin{picture}(18,9)(0,-5)
      \gasset{Nadjust=wh,Nadjustdist=0.5,Nfill=n,Nframe=n,AHnb=0,linewidth=.1}
      \node(a)(0,3){$c$}
      \node(b)(9,3){$a$}
      \node(d)(4.5,-3){$e$}
      \node(c)(13.5,-3){$d$}
      \node(e)(22.5,-3){$b$}
      \drawedge(a,b){}
      \drawedge(b,c){}
      \drawedge(d,c){}
      \drawedge(e,c){}
      \drawedge(b,d){}
    \end{picture}
  \end{center}
  Then the corresponding dependence alphabet is:
  \begin{center}
    \setlength{\unitlength}{1.2mm}
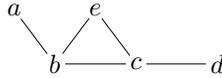
    \begin{picture}(18,9)(0,-5)
      \gasset{Nadjust=wh,Nadjustdist=0.5,Nfill=n,Nframe=n,AHnb=0,linewidth=.1}
      \node(b)(0,3){$a$}
      \node(d)(9,3){$e$}
      \node(e)(4.5,-3){$b$}
      \node(a)(13.5,-3){$c$}
      \node(c)(22.5,-3){$d$}
      \drawedge(e,b){}
      \drawedge(e,a){}
      \drawedge(a,c){}
      \drawedge(a,d){}
      \drawedge(e,d){}
    \end{picture}
  \end{center}
  We consider the words $u=aeadbacdd$ and $v=eaabdcaeb$. Then the
  dependence graphs $D_u$ of $u$ and $D_v$ of $v$ look as follows
  (where we label the vertices $i$ with the letter $u[i]$
  (resp. $v[i]$)):
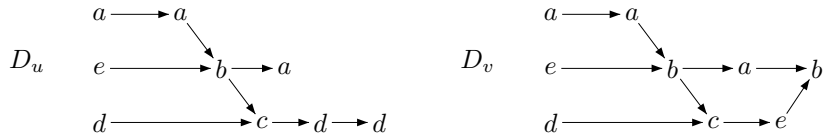
  \begin{center}
    \setlength{\unitlength}{1.2mm}
    \begin{picture}(78,15)(0,-10)
      \gasset{Nadjust=wh,Nadjustdist=0.5,Nfill=n,Nframe=n,AHnb=1,linewidth=.1}
      \put(-10,-3){$D_u$}
      \node(b1)(0,3){$a$}
      \node(d1)(0,-3){$e$}
      \node(c1)(0,-9){$d$}
      \node(b2)(9,3){$a$}
      \node(e1)(13.5,-3){$b$}
      \node(b3)(20.5,-3){$a$}
      \node(a1)(18,-9){$c$}
      \node(c2)(24.5,-9){$d$}
      \node(c3)(31,-9){$d$}
      \drawedge(b1,b2){}
      \drawedge(d1,e1){}
      \drawedge(b2,e1){}
      \drawedge(c1,a1){}
      \drawedge(e1,a1){}
      \drawedge(e1,b3){}
      \drawedge(a1,c2){}
      \drawedge(c2,c3){}
      \put(40,-3){$D_v$}
      \node(b1')(50,3){$a$}
      \node(d1')(50,-3){$e$}
      \node(c1')(50,-9){$d$}
      \node(b2')(59,3){$a$}
      \node(e1')(63.5,-3){$b$}
      \node(b3')(71.5,-3){$a$}
      \node(a1')(68,-9){$c$}
      \node(e2')(79.5,-3){$b$}
      \node(d2')(75.5,-9){$e$}
      \drawedge(b1',b2'){}
      \drawedge(d1',e1'){}
      \drawedge(b2',e1'){}
      \drawedge(c1',a1'){}
      \drawedge(e1',a1'){}
      \drawedge(e1',b3'){}
      \drawedge(b3',e2'){}
      \drawedge(a1',d2'){}
      \drawedge(d2',e2'){}
    \end{picture}
  \end{center}
  Then we have $u\sqcap v=aeadbac=:p$ and its dependence graph is:

  \begin{center}
    \setlength{\unitlength}{1.2mm}
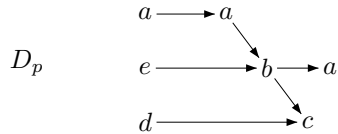
    \begin{picture}(28,15)(0,-10)
      \gasset{Nadjust=wh,Nadjustdist=0.5,Nfill=n,Nframe=n,AHnb=1,linewidth=.1}
      \put(-15,-3){$D_p$}
      \node(b1)(0,3){$a$}
      \node(d1)(0,-3){$e$}
      \node(c1)(0,-9){$d$}
      \node(b2)(9,3){$a$}
      \node(e1)(13.5,-3){$b$}
      \node(b3)(20.5,-3){$a$}
      \node(a1)(18,-9){$c$}
      \drawedge(b1,b2){}
      \drawedge(d1,e1){}
      \drawedge(b2,e1){}
      \drawedge(c1,a1){}
      \drawedge(e1,a1){}
      \drawedge(e1,b3){}
    \end{picture}
  \end{center}
  Since $u\setminus p= dd$ and $v\setminus p=eb$ we have $(u\setminus
  p) I (v\setminus p)$ and hence the supremum
  $s=u\sqcup v= aeadbacddeb$ is defined. The dependence graph for $s$ is:\\
  \vspace{1cm}
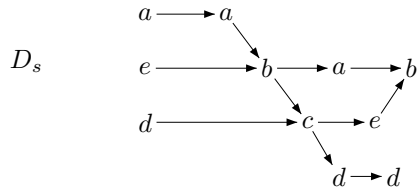
  \begin{center}
    \setlength{\unitlength}{1.2mm}
    \begin{picture}(28,15)(0,-17)
      \gasset{Nadjust=wh,Nadjustdist=0.5,Nfill=n,Nframe=n,AHnb=1,linewidth=.1}
      
      \put(-15,-3){$D_s$}
      
      \node(b1')(0,3){$a$}
      \node(d1')(0,-3){$e$}
      \node(c1')(0,-9){$d$}
      \node(b2')(9,3){$a$}
      \node(e1')(13.5,-3){$b$}
      \node(b3')(21.5,-3){$a$}
      \node(a1')(18,-9){$c$}
      \node(e2')(29.5,-3){$b$}
      \node(d2')(25.5,-9){$e$}
      \node(c2')(21.5,-15){$d$}
      \node(c3')(27.5,-15){$d$}

      \drawedge(b1',b2'){}
      \drawedge(d1',e1'){}
      \drawedge(b2',e1'){}
      \drawedge(c1',a1'){}
      \drawedge(e1',a1'){}
      \drawedge(e1',b3'){}
      \drawedge(b3',e2'){}
      \drawedge(a1',d2'){}
      \drawedge(d2',e2'){}
      \drawedge(a1',c2'){}
      \drawedge(c2',c3'){}
    \end{picture}
  \end{center}
\end{example}
The following lemma is a basic statement for traces, see for example
\cite{DieRoz95}:

\begin{lemma}[Levi's Lemma]\label{levi_lemma}
  Let $u_1,u_2,v_1,v_2$ be traces such that $u_1u_2=v_1v_2$. Then
  there exist traces $x,y_1,y_2,z$ such that $y_1Iy_2$ and $u_1=xy_1$,
  $u_2=y_2z$, $v_1=xy_2$, and $v_2=y_1z$.
\end{lemma}
We use Levi's Lemma to prove the following statement:

\begin{lemma}\label{unique_decomposition}
  Let $a\in \Sigma$. The decomposition of a trace $t\in\dM(\Sigma,I)$
  as $t=u_1u_2$ with $u_2Ia$ and $|u_2|$ maximal is unique in
  $\dM(\Sigma,I)$.
\end{lemma}

\begin{proof}
  Let $u_1u_2=t=v_1v_2$ be such that $u_2Ia$, $v_2Ia$ and $|u_2|$ and
  $|v_2|$ are both maximal (hence $|u_2|=|v_2|$). By Levi's Lemma
  there are traces $x,y_1,y_2,z$ such that $y_1Iy_2$ and $u_1=xy_1$,
  $u_2=y_2z$, $v_1=xy_2$, and $v_2=y_1z$.  From $u_2Ia$ and $v_2Ia$ we
  get $y_1Ia$ and $y_2Ia$.  Maximality of $|u_2|=|v_2|$ and
  $xy_1u_2=t=xy_2v_2$ implies $y_1=y_2=\varepsilon$. Hence $u_1=v_1$
  and $u_2=v_2$.  \qed
\end{proof}
A \emph{trace rewriting system} $R$ over $\mathbb{M}(\Sigma,I)$ is
just a finite subset of $\mathbb{M}(\Sigma,I) \times
\mathbb{M}(\Sigma,I)$ \cite{Die90lncs}. We can define the
\emph{one-step rewrite relation} $\to_R \;\subseteq
\mathbb{M}(\Sigma,I) \times \mathbb{M}(\Sigma,I)$ by: $x \to_R y$ if
and only if there are $u,v \in \mathbb{M}(\Sigma,I)$ and $(\ell,r) \in
R$ such that $x = u\ell v$ and $y = u r v$. With $\xrightarrow{*}_R$
we denote the reflexive transitive closure of $\rightarrow_R$. The
notion of a confluent and terminating trace rewriting system is
defined as for other types of rewriting systems \cite{BoOt93}: A trace
rewriting system $R$ is called \emph{confluent} if for all $u,v,v'\in
\mathbb{M}(\Sigma,I)$ it holds that $u \xrightarrow{*}_R v$ and
$u\xrightarrow{*}_R v'$ imply that there is a trace $w$ with $v
\xrightarrow{*}_R w$ and $v'\xrightarrow{*}_R w$.  It is called
\emph{terminating} if there does not exist an infinite chain
$u_0\rightarrow_R u_1 \rightarrow_R u_2 \cdots$.  A trace $u$ is
\emph{$R$-irreducible} if no trace $v$ with $u \to_R v$ exists. The
set of all $R$-irreducible traces is denoted
with $\IRR(R)$. If $R$ is terminating and confluent, then for every
trace $u$, there exists a unique \emph{normal form} $\NF_R(u) \in
\IRR(R)$ such that $u \xrightarrow{*}_R \NF_R(u)$. 

Let us now work in the trace monoid $\dM(\Sigma^{\pm 1}, I)$.
For a trace $u=[a_1\cdots a_n]_I\in\dM(\Sigma^{\pm 1},I)$ 
we denote with $u^{-1}$ the trace
$u^{-1}=[a_n^{-1}\cdots a_1^{-1}]_I$.  It is easy to see that this
definition is independent of the chosen representative $a_1\cdots a_n$
of the trace $u$.  It follows that we have
$[\val(\dA)]_I^{-1}=[\val(\dA^{-1})]_I$ for an SLP $\dA$.  For the
rest of the paper, we fix the trace rewriting system
$$
R=\{([aa^{-1}]_I,[\epsilon]_I)\mid  a\in \Sigma^{\pm 1}\}
$$
over the trace monoid $\dM(\Sigma^{\pm 1},I)$.  Since $R$ is
length-reducing, $R$ is terminating.  By \cite{Die90lncs,Wra88}, $R$
is also confluent.  For traces $u,v \in \dM(\Sigma^{\pm 1},I)$ we have
$u=v$ in $\dG(\Sigma,I)$ if and only if $\NF_R(u)=\NF_R(v)$.  Using
these facts, it was shown in \cite{Die90lncs,Wra88} that the word
problem for $\dG(\Sigma,I)$ can be solved in linear time (on the RAM
model).

\section{Algorithms for compressed traces}

In this section, we will recall some results from \cite{LoSchl07}
concerning traces, which are represented by SLPs.  For SLPs $\dA$ and
$\dB$ over $\Sigma^{\pm 1}$ we say that $\dB$ is an
\emph{$R$-reduction} of $\dA$ if $[\val(\dB)]_I=\NF_R([\val(\dA)]_I)$.
We will need the following theorem.

\begin{theorem}[\cite{LoSchl07}]\label{R-reduction}
  Let $\dA$ be an SLP over $\Sigma^{\pm 1}$ representing a trace in
  $\dM(\Sigma^{\pm 1},I)$. We can compute an $R$-reduction for $\dA$
  in polynomial time.
\end{theorem}

\begin{corollary}\label{cwp_graphgroup}
  The following decision problem can be solved in polynomial time.

  \medskip
  \noindent
  INPUT: An SLP $\dA$ over $\Sigma^{\pm 1}$.\\
  QUESTION: $\NF_R([\val(\dA)]_I)=[\epsilon]_I$?
\end{corollary}
Note that this is equivalent to a polynomial time solution of the
compressed word problem for graph groups.

\begin{theorem}[\cite{LoSchl07}] \label{theo-compressed-inf} 
  For given SLPs $\dA_0$ and $\dA_1$ over $\Sigma^{\pm 1}$, we can compute in
  polynomial time SLPs $\dP$, $\dD_0$, $\dD_1$ with
  $[\val(\dP)]_I=[\val(\dA_0)]_I\sqcap[\val(\dA_1)]_I$ and
  $[\val(\dD_i)]_I=[\val(\dA_i)]_I\setminus [\val(\dA_{1-i})]_I$ ($i
  \in \{0,1\}$).
\end{theorem}
An immediate corollary of Theorem~\ref{theo-compressed-inf} and
Lemma~\ref{lemma-supremum-exists} is:

\begin{corollary} \label{theo-compressed-sup} For given SLPs $\dA_0$
  and $\dA_1$ over $\Sigma^{\pm 1}$, we can check in polynomial time, whether
  $[\val(\dA_0)]_I\sqcup[\val(\dA_1)]_I$ exists, and in case it
  exists, we can compute in polynomial time an SLP $\dS$ with
  $[\val(\dS)]_I = [\val(\dA_0)]_I\sqcup[\val(\dA_1)]_I$.
\end{corollary}
Lemma~\ref{trace_supremum} and Corollary~\ref{theo-compressed-sup}
imply the following corollary.

\begin{corollary}\label{decide_compressed_supremum}
  Let $r$ be a fixed constant.  For given SLPs $\dV_1,\dots,\dV_r$
  over $\Sigma^{\pm 1}$, we can decide in polynomial time whether
  $[\val(\dV_1)]_I\sqcup\dots\sqcup[\val(\dV_r)]_I$ exists, and in
  case it exists we can compute in polynomial time an SLP $\dS$ with
  $[\val(\dS)]_I=[\val(\dV_1)]_I\sqcup\dots\sqcup[\val(\dV_r)]_I$.
\end{corollary}
It is important that we fix the number $r$ of SLPs in
Corollary~\ref{decide_compressed_supremum}: Each application 
of Lemma~\ref{theo-compressed-sup} may increase the size of the SLP
polynomially. Hence, a non-fixed number of applications might
lead to an exponential blow-up.

\section{Double $a$-cones}

The definition of the problem $\RSCCP(\Sigma,I)$ 
in Section~\ref{sec:main-results} motivates 
the following definition: A \emph{double $a$-cone} for $a\in\Sigma^{\pm 1}$ is an
$R$-irreducible trace of the form $uau^{-1}$ with $u\in
\dM(\Sigma^{\pm 1},I)$.  In this section, we will prove several
results on double $a$-cones, which will be needed later for deciding
$\RSCCP(\Sigma,I)$ in polynomial time.

\begin{lemma}\label{char_cone}
  A trace $uau^{-1}$ is a double $a$-cone if and only if the following
  conditions hold:
  \begin{enumerate}[(1)]
  \item $u \in \IRR(R)$
  \item  $\max(u)\cap (\{a,a^{-1}\}\cup I(a)) =\emptyset$.
  \end{enumerate}
\end{lemma}

\begin{proof}
  Let $v=uau^{-1}$ be a double $a$-cone. 
  Since $v \in \IRR(R)$, also $u \in \IRR(R)$.
  If $a^\epsilon\in \max(u)$ for
  $\epsilon\in\{1,-1\}$ then $v=u'a^{\epsilon} a
  a^{-\epsilon} u'^{-1}$ for
  some trace $u'$ contradicting the $R$-irreducibility of
  $v$. Similarly, if there is some $b\in I(a)\cap \max(u)$ it follows
  that $v=u'bab^{-1}u'^{-1}=u'abb^{-1}u'^{-1}$ again a
  contradiction. Suppose on the other hand that $(1)$ and $(2)$ hold
  for $v=uau^{-1}$. Since $u \in \IRR(R)$ and no element from
  $\max(u)$ cancels against or commutes with $a$ it follows that $v$
  is also $R$-irreducible.  \qed
\end{proof}
It follows that every letter in a double $a$-cone either lies before
or after the central letter $a$. Its dependence graph always has the
following form:
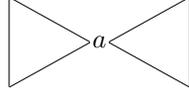
\begin{center}
  \setlength{\unitlength}{1.2mm}
  \begin{picture}(20,10)(-10,-5)
    \gasset{Nadjust=wh,Nadjustdist=0,Nfill=n,Nframe=n,AHnb=0,linewidth=.1}
    \node(a)(0,0){$a$}
    \node(lo)(-10,5){}
    \node(lm)(-1,0){}
    \node(lu)(-10,-5){}
    \node(ro)(10,5){}
    \node(rm)(1,0){}
    \node(ru)(10,-5){}
    \drawedge(lo,lm){}
    \drawedge(lo,lu){}
    \drawedge(lm,lu){}
    \drawedge(ro,rm){}
    \drawedge(ro,ru){}
    \drawedge(rm,ru){}
  \end{picture}
\end{center}
By the following lemma, each double $a$-cone has a unique
factorization of the form $u_1 b u_2$ with $|u_1|=|u_2|$.

\begin{lemma}\label{unique_cone}
  Let $v=uau^{-1}$ be a double $a$-cone and let $v=u_1bu_2$ with
  $b\in\Sigma^{\pm 1}$ and $|u_1|=|u_2|$. Then $a=b$, $u_1=u$ and
  $u_2=u^{-1}$.
\end{lemma}
	
\begin{proof}
  Let $v=uau^{-1}=u_1bu_2$ be a double $a$-cone where
  $|u_1|=|u_2|$. We have $\max(ua)=\{a\}$ and $(a,c)\in D$ for all
  $c\in \min(u^{-1})$. Moreover $|u_1|=|u_2|=|u|$. By Levi's Lemma,
  there exist traces $x$, $y_1$, $y_2$ and $z$ with $u_1b=xy_1$,
  $u_2=y_2z$, $ua=xy_2$ and $u^{-1}=y_1z$. Assume that
  $y_2\neq\varepsilon$. Since $\max(y_2)\subseteq \max(ua)=\{a\}$ we
  get $\max(y_2)=\{a\}$. Since $(a,c)\in D$ for all $c\in
  \min(y_1)\subseteq \min(u^{-1})$ and $y_1Iy_2$ it follows
  $y_1=\varepsilon$. But then $|u|=|u^{-1}|=|z|<|y_2z|=|u_2|$ leads to
  a contradiction. Hence, we must have $y_2=\varepsilon$. Thus
  $|u|=|u^{-1}|=|y_1z|=|y_1|+|z|=|y_1|+|u_2|=|y_1|+|u|$ implies
  $y_1=\varepsilon$. Therefore we get $ua=u_1b$ and
  $u^{-1}=u_2$. Finally, since $\max(ua)=\{a\}$ we must have $a=b$ and
  $u=u_1$.  \qed
\end{proof}

\begin{lemma} \label{lemma-double-a-cone} 
  Let $w\in \dM(\Sigma^{\pm 1}, I)$ be $R$-irreducible and $a\in
  \Sigma^{\pm 1}$. Then the following three conditions are equivalent:
  \begin{enumerate}[(1)]
  \item There exists $x\in \dM(\Sigma^{\pm 1}, I)$ with $w=xax^{-1}$ in $\dG(\Sigma, I)$.
  \item There exists $x\in \dM(\Sigma^{\pm 1}, I)$ with $w=xax^{-1}$ in $\dM(\Sigma^{\pm 1}, I)$.
  \item  $w$ is a double $a$-cone.
  \end{enumerate}
\end{lemma}

\begin{proof}
  Direction ``$(2)\Rightarrow (1)$'' is trivially true and
  ``$(2)\Leftrightarrow (3)$'' is just the definition of a double
  $a$-cone.  For ``$(1)\Rightarrow (2)$'' assume that $w=xax^{-1}$ in
  $\dG(\Sigma, I)$. Since $w \in \IRR(R)$, we have
  $xax^{-1}\xrightarrow{*}_R w$. W.l.o.g., $x \in \IRR(R)$. Let
  $n\geq 0$ such that $xax^{-1}\xrightarrow{}^n_R w$. We prove (2) by
  induction on $n$.

  For $n=0$ we have $w=xax^{-1}$ in $\dM(\Sigma^{\pm 1}, I)$. So
  assume $n>0$. If $a\in \max(x)$ we have $x=ya$ for some
  $y\in\dM(\Sigma^{\pm 1}, I)$ and hence $xax^{-1}=
  yaaa^{-1}y^{-1}\rightarrow_R yay^{-1}$. Since $R$ is confluent, we
  have $yay^{-1}\xrightarrow{*}_R w$ and since each rewriting rule
  from $R$ reduces the length of a trace by $2$ it follows that
  $yay^{-1}\xrightarrow{}^{n-1}_R w$. Hence, by induction, there
  exists a trace $v$ with $w=vav^{-1}$ in $\dM(\Sigma^{\pm 1},
  I)$. The case where $a^{-1}\in\max(x)$ is analogous to the previous
  case. If there exists $b\in\max(x)$ with $(a,b)\in I$ we can infer
  that $x=yb$ for some trace $y$ and
  $xax^{-1}=ybab^{-1}y^{-1}=yabb^{-1}y^{-1}\rightarrow_R yay^{-1}$. As
  for the previous cases we obtain inductively $w=vav^{-1}$ in
  $\dM(\Sigma^{\pm 1},I)$ for some trace $v$. Finally, if
  $\max(x)\cap(\{a,a^{-1}\}\cup I(a))=\emptyset$, then $xax^{-1}$ is a
  double $a$-cone by Lemma~\ref{char_cone} and hence $R$-irreducible,
  which contradicts $n>0$.  \qed
\end{proof}

\begin{lemma}\label{condition_if_conjugating_then_supremum}
  Let $w_a,v_a \in \dM(\Sigma^{\pm 1},I)$ ($a \in \Sigma$) be
  $R$-irreducible such that $w_a=v_a a v_a^{-1}$ in $\dM(\Sigma^{\pm
    1},I)$ for all $a \in \Sigma$ (thus, every $w_a$ is a double
  $a$-cone).  If there is a trace $x\in \dM(\Sigma^{\pm 1},I)$ with $x
  a x^{-1}=w_a$ in $\dG(\Sigma,I)$ for all $a \in \Sigma$, then $s =
  \bigsqcup_{a\in\Sigma} v_a$ exists and $s a s^{-1}=w_a$ in
  $\mathbb{G}(\Sigma,I)$ for all $a \in \Sigma$.
\end{lemma}

\begin{proof}
  Assume that a trace $x\in\dM(\Sigma^{\pm 1},I)$ exists with
  $xax^{-1}=w_a$ in $\mathbb{G}(\Sigma,I)$ for all $a\in\Sigma$.  We
  can assume w.l.o.g. that $x \in \IRR(R)$. First, write $x$ as
  $x=x_a a^{n_a}$ with $n_a\in \mathbb{Z}$ and $|n_a|$ maximal for
  every $a \in \Sigma$.  Then $a, a^{-1}\not\in \max(x_a)$ and $x_a$
  is uniquely determined by the cancellativity of $\dM(\Sigma^{\pm
    1},I)$. Next we write $x_a$ as $x_a=t_a u_a$ with $u_a I a$ and
  $|u_a|$ maximal.  This decomposition is unique by
  Lemma~\ref{unique_decomposition}. We get
  $$
  xax^{-1}=t_a u_a a^{n_a} a a^{-n_a}u_a^{-1}t_a^{-1}\xrightarrow{*}_R
  t_a u_a a u_a^{-1}t_a^{-1}\xrightarrow{*}_R t_a a t_a^{-1}
  \xrightarrow{*}_R w_a = v_a a v_a^{-1} .
  $$
  From the choice of $n_a$ and $u_a$ it follows that $\max(t_a)\cap (\{a,
  a^{-1}\} \cup I(a)) =\emptyset$. This implies that $t_a a t_a^{-1}
  \in \IRR(R)$.
  Hence $t_a a t_a^{-1}=v_a a v_a^{-1}$ in
  $\dM(\Sigma^{\pm 1},I)$ and by Lemma~\ref{unique_cone} it follows
  that $t_a=v_a$. So for all $a\in\Sigma$ it holds that $v_a\preceq x$
  and therefore $s= \bigsqcup_{a\in\Sigma} v_a$ exists.

  Now we infer that $s a s^{-1}=w_a$ in $\mathbb{G}(\Sigma,I)$ for all
  $a\in\Sigma$. Since $v_a\preceq x$ for all $a\in\Sigma$, there is
  some trace $y$ such that $x=sy$ in $\mathbb{M}(\Sigma^{\pm 1},I)$. 
  We can write $s=v_a r_a$ 
  for all $a\in\Sigma$.  Let $z_a = r_a
  y$ and hence $x = v_a r_a y = v_a z_a$. 
   As a suffix of the $R$-irreducible trace $x$, $z_a$ is
  $R$-irreducible as well.  By assumption we have
  $$
  \forall a\in\Sigma : v_a a v_a^{-1} = w_a = xax^{-1} = v_a z_a a z_a^{-1} v_a^{-1}
  $$ 
  in $\mathbb{G}(\Sigma,I)$ and hence, by cancelling $v_a$ and
  $v_a^{-1}$,
  $$
  \forall a\in\Sigma :  a  = z_a  a z_a^{-1}
  $$ 
  in $\mathbb{G}(\Sigma,I)$. Since $a$ as a single symbol is
  $R$-irreducible, this means that
  \begin{equation} \label{derivation to _a} \forall a\in\Sigma : z_a a
    z_a^{-1} \to^*_R a .
  \end{equation}
  We prove by induction on $|z_a|$ that $\alp(z_a) \subseteq I(a) \cup \{a,a^{-1}\}$.  
  The case $z_a = \varepsilon$ is
  clear. Now assume that $z_a \neq \varepsilon$.  If every maximal
  symbol in $z_a$ belongs to $\Sigma^{\pm 1} \setminus (I(a) \cup
  \{a,a^{-1}\})$, then $z_a a z_a^{-1} \in \IRR(R)$ (recall
  that $z_a \in \IRR(R)$), which contradicts (\ref{derivation to _a}). 
  Hence, let $z_a = z'_a b$ with $b \in I(a) \cup
  \{a,a^{-1}\}$.  We get $z_a a z_a^{-1} \to_R z'_a a {z'}_a^{-1}
  \to^*_R a$. By induction, it follows that $\alp(z'_a)
  \subseteq I(a) \cup \{a,a^{-1}\}$. Hence, the same is true for
  $z_a$ and therefore for the prefix $r_a$ of $z_a$ as well.  But this
  implies $sas^{-1}= v_a r_a a r_a^{-1} v_a^{-1} = v_a a v_a^{-1} =
  w_a$ in $\mathbb{G}(\Sigma,I)$.  \qed
\end{proof}

\section{Restricted simultaneous compressed conjugacy}
\label{sec:RSSC}

Based on our results on double $a$-cones from the previous section, we
will prove Theorem~\ref{decide_ccp} in this section. First, we have to
prove the following lemma:

\begin{lemma}\label{slp_is_cone}
  Let $a\in \Sigma^{\pm 1}$.  For a given SLP $\dA$ with
  $[\val(\dA)]_I \in \IRR(R)$, we can check in polynomial time
  whether $[\val(\dA)]_I$ is a double $a$-cone.  In case
  $[\val(\dA)]_I$ is a double $a$-cone, we can compute in polynomial
  time an SLP $\dV$ over $\Sigma^{\pm 1}$ with
  $\val(\dA)=\val(\dV)\,a\,\val(\dV^{-1})$ in $\dM(\Sigma^{\pm 1},I)$.
\end{lemma}

\begin{proof}
  First we check whether $|\val(\dA)|$ is odd.  If not, then
  $[\val(\dA)]_I$ cannot be a double $a$-cone. Assume that
  $|\val(\dA)|=2k+1$ for some $k\geq 0$ and let $\val(\dA)=u_1bu_2$
  with $|u_1|=|u_2|=k$. By \cite{Hag00} we can construct SLPs $\dV_1$
  and $\dV_2$ such that $\val(\dV_1)=\val(\dA)[1:k] =u_1$ and
  $\val(\dV_2)=\val(\dA)[k+2:2k+1]=u_2$.  By Lemma~\ref{unique_cone},
  $[\val(\dA)]_I$ is a double $a$-cone if and only if $a=b$ and
  $[\val(\dV_1)]_I=[\val(\dV_2^{-1})]_I$. This can be checked in
  polynomial time.  \qed
\end{proof}
Now we are in the position to present a polynomial time algorithm for
$\RSCCP(\Sigma,I)$:

\medskip

\noindent
{\it Proof of Theorem \ref{decide_ccp}.}
Let $\dA_a$ ($a \in \Sigma$) be the input SLPs. We have to check
whether there exists $x$ such that $\val(\dA_a) = xax^{-1}$ 
in $\mathbb{G}(\Sigma,I)$ for all $a
\in \Sigma$.  Since the SLP $\dA_a$ and an $R$-reduction of $\dA_a$
represent the same group element in $\mathbb{G}(\Sigma,I)$, Theorem
\ref{R-reduction} allows us to assume that the input SLPs $\dA_a$ ($a
\in \Sigma$) represent $R$-irreducible traces.

We first check whether every trace $[\val(\dA_a)]_I$ is a double
$a$-cone.  By Lemma~\ref{slp_is_cone} this is possible in polynomial
time.  If there exists $a \in \Sigma$ such that $[\val(\dA_a)]_I$ is
not a double $a$-cone, then we can reject by
Lemma~\ref{lemma-double-a-cone}. Otherwise, we can compute (using
again Lemma~\ref{slp_is_cone}) SLPs $\dV_a$ ($a \in \Sigma$) such that
$[\val(\dA_a)]_I = [\val(\dV_a)]_I a [\val(\dV_a)]_I^{-1}$ in
$\dM(\Sigma^{\pm 1},I)$.  Finally,
by Lemma~\ref{condition_if_conjugating_then_supremum}, it suffices to
check whether $\bigsqcup_{a\in\Sigma} [\val(\dV_a)]_I$
exists, which is possible in polynomial time by
Corollary~\ref{decide_compressed_supremum} (recall that $|\Sigma|$ is
a constant in our consideration). Moreover, if this supremum exists,
then we can compute in polynomial time an SLP $\dS$ with
$[\val(\dS)]_I = \bigsqcup_{a\in\Sigma} [\val(\dV_a)]_I$.  Then,
$\val(\dS)$ is a solution for our $\RSCCP(\Sigma,I)$-instance.  \qed

\section{Computing the core of a compressed trace}

In order to prove Theorem~\ref{decidesccp} we need some further
concepts from \cite{Wra88}.

\begin{definition}
  A trace $y$ is called \emph{cyclically $R$-irreducible} if $y \in
  \IRR(R)$  and $\min(y)\cap\min(y^{-1})=\emptyset$. If for a trace
  $x$ we have $\NF_R(x) = uyu^{-1}$ in $\dM(\Sigma^{\pm 1},I)$ for
  traces $y,u$ with $y$ cyclically $R$-irreducible, then we call $y$ the
  \emph{core} of $x$, $\CR(x)$ for short.
\end{definition}
The trace $y$ in the last definition is  uniquely defined
\cite{Wra88}.  Moreover, note that a trace $t$ is a double $a$-cone if and
only if $t \in \IRR(R)$ and $\CR(t) = a$.

In this section, we will present a polynomial time algorithm for
computing an SLP that represents $\CR([\val(\dA)]_I)$ for a given SLP
$\dA$.  For this, we need the following lemmas.

\begin{lemma}\label{prefixinverseepsilon}
  Let $p,t\in\dM(\Sigma^{\pm 1},I)$. If $p\preceq t$, $p^{-1}\preceq
  t$ and $t \in \IRR(R)$, then $p=\varepsilon$.
\end{lemma}

\begin{proof}
  Suppose for contradiction that
  $$
  T=\{t \in \IRR(R) \mid \exists p\in \dM(\Sigma^{\pm
    1},I) \setminus \{ \varepsilon\} : p\preceq t \wedge p^{-1}\preceq
  t \} \neq\emptyset.
  $$ 
  Let $t\in T$ with $|t|$ minimal and $p\in \dM(\Sigma^{\pm 1},I)$
  such that $p \neq \varepsilon$, $p\preceq t$, and $p^{-1}\preceq
  t$. If $|p|=1$ then $p=a$ for some $a\in\Sigma^{\pm 1}$ and hence
  $a\preceq t$ and $a^{-1}\preceq t$, a contradiction since
  $aDa^{-1}$. If $|p|=2$ then $p=a_1a_2$ for some
  $a_1,a_2\in\Sigma^{\pm 1}$. Since $t$, and therefore $p$ is
  $R$-irreducible, we have $a_1 \neq a_2^{-1}$. Since $a_1\in \min(t)$ and
  $a_2^{-1}\in\min(t)$ we have $a_1 I a_2^{-1}$, i.e., $a_1 I
  a_2$. Hence, also $a_2 \in \min(t)$, which contradicts
  $a_2^{-1}\in\min(t)$. So assume that $|p|>2$. Let $a\in
  \min(p)$. Then $a\in \min(t)$, and there exist traces $y, t'$ with
  $t = at' = p^{-1} y$.  If $a \not\in \min(p^{-1})$, then $a \in
  \min(y)$ and $a I p^{-1}$. But the latter independence contradicts
  $a^{-1} \in \alp(p^{-1})$. Hence $a \in \min(p^{-1})$, i.e., $a^{-1}
  \in \max(p)$. Thus, we can write $p=aqa^{-1}$ and
  $p^{-1}=aq^{-1}a^{-1}$ with $q \neq \varepsilon$.  Since $aqa^{-1} =
  p \preceq at'$, $aq^{-1}a^{-1} = p^{-1} \preceq at'$ and
  $\dM(\Sigma^{\pm 1},I)$ is cancellative, we have a $q\preceq t'$,
  $q^{-1}\preceq t'$. Since $q \neq \varepsilon$, we have a
  contradiction to the fact that $|t|$ is minimal.  \qed
\end{proof}

\begin{example}\label{exacore}
We take the independence alphabet from Example \ref{Ex traces1} and
consider the trace
$x=[c^{-1}d^{-1}a^{-1}ba^{-1}cabdc^{-1}d^{-1}a^{-1}b^{-1}dca]_I\in\dM(\Sigma^{\pm
1},I)$, whose dependence graph looks as follows:
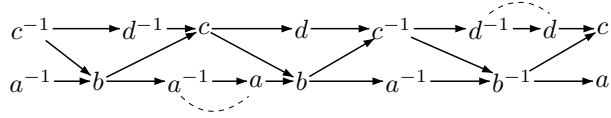
\begin{center}
\setlength{\unitlength}{1mm}
\begin{picture}(76,13)(-6,-3)
\gasset{Nadjust=wh,Nadjustdist=0.3,Nfill=n,Nframe=n,AHnb=1,linewidth=.2}
\node(oc1)(-6,7){$c^{-1}$}
\node(od1)(9,7){$d^{-1}$}
\node(oc2)(17,7){$c$}
\node(od2)(30,7){$d$}
\node(oc3)(42,7){$c^{-1}$}
\node(od3)(55,7){$d^{-1}$}
\node(od4)(63,7){$d$}
\node(oc4)(70,7){$c$}
\node(ua1)(-6,0){$a^{-1}$}
\node(ub1)(3,0){$b$}
\node(ua2)(15,0){$a^{-1}$}
\node(ua3)(24,0){$a$}
\node(ub2)(30,0){$b$}
\node(ua4)(44,0){$a^{-1}$}
\node(ub3)(58,0){$b^{-1}$}
\node(ua5)(70,0){$a$}

\drawedge(oc1,ub1){}
\drawedge(ub1,oc2){}
\drawedge(oc2,ub2){}
\drawedge(ub2,oc3){}
\drawedge(oc3,ub3){}
\drawedge(ub3,oc4){}

\drawedge(oc1,od1){}
\drawedge(od1,oc2){}
\drawedge(oc2,od2){}
\drawedge(od2,oc3){}
\drawedge(oc3,od3){}
\drawedge(od3,od4){}
\drawedge(od4,oc4){}
\drawedge(ua1,ub1){}
\drawedge(ub1,ua2){}
\drawedge(ua2,ua3){}
\drawedge(ua3,ub2){}
\drawedge(ub2,ua4){}
\drawedge(ua4,ub3){}
\drawedge(ub3,ua5){}
\drawqbezier[AHnb=0,linewidth=.1,dash={.7}0](14,-2,18,-6,23,-2)
\drawqbezier[AHnb=0,linewidth=.1,dash={.7}0](55,9,59,12,63,9)
\end{picture}
\end{center}
Then the $R$-reduction of $x$ is
$\NF_R(x)=[c^{-1}d^{-1}a^{-1}bcbdc^{-1}a^{-1}b^{-1}ca]_I$:
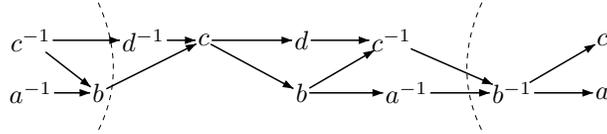
\begin{center}
\setlength{\unitlength}{1mm}
\begin{picture}(76,13)(-6,-3)
\gasset{Nadjust=wh,Nadjustdist=0.3,Nfill=n,Nframe=n,AHnb=1,linewidth=.2}
\node(oc1)(-6,7){$c^{-1}$}
\node(od1)(9,7){$d^{-1}$}
\node(oc2)(17,7){$c$}
\node(od2)(30,7){$d$}
\node(oc3)(42,7){$c^{-1}$}
\node(oc4)(70,7){$c$}
\node(ua1)(-6,0){$a^{-1}$}
\node(ub1)(3,0){$b$}
\node(ub2)(30,0){$b$}
\node(ua4)(44,0){$a^{-1}$}
\node(ub3)(58,0){$b^{-1}$}
\node(ua5)(70,0){$a$}
\drawedge(oc1,ub1){}
\drawedge(ub1,oc2){}
\drawedge(oc2,ub2){}
\drawedge(ub2,oc3){}
\drawedge(oc3,ub3){}
\drawedge(ub3,oc4){}
\drawedge(oc1,od1){}
\drawedge(od1,oc2){}
\drawedge(oc2,od2){}
\drawedge(od2,oc3){}
\drawedge(ua1,ub1){}
\drawedge(ub2,ua4){}
\drawedge(ua4,ub3){}
\drawedge(ub3,ua5){}
\drawqbezier[AHnb=0,linewidth=.1,dash={.7}0](3,12,7,3,3,-5)
\drawqbezier[AHnb=0,linewidth=.1,dash={.7}0](54,12,50,3,54,-5)
\end{picture}
\end{center}
Hence, the core of $x$ is $\CR(x)=[d^{-1}cbdc^{-1}a^{-1}]_I$ and looks
as follows:
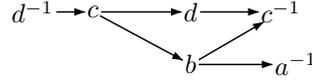
\begin{center}
\setlength{\unitlength}{1mm}
\begin{picture}(35,13)(9,-3)
\gasset{Nadjust=wh,Nadjustdist=0.3,Nfill=n,Nframe=n,AHnb=1,linewidth=.2}
\node(od1)(9,7){$d^{-1}$}
\node(oc2)(17,7){$c$}
\node(od2)(30,7){$d$}
\node(oc3)(42,7){$c^{-1}$}
\node(ub2)(30,0){$b$}
\node(ua4)(44,0){$a^{-1}$}
\drawedge(oc2,ub2){}
\drawedge(ub2,oc3){}
\drawedge(od1,oc2){}
\drawedge(oc2,od2){}
\drawedge(od2,oc3){}
\drawedge(ub2,ua4){}
\end{picture}
\end{center}
Note that we have
$\NF_R (x)\sqcap \NF_R(x^{-1})=c^{-1}a^{-1}b$
  and hence
  \begin{align*}
    \NF_R\Big(\big(\NF_R(x)\sqcap& \NF_R(x^{-1})\big)^{-1}\NF_R(x)\big(\NF_R(x)\sqcap \NF_R(x^{-1}\big)\Big)\\
    &=\NF_R\Big(\big(c^{-1}a^{-1}b\big)^{-1}\big(c^{-1}d^{-1}a^{-1}bcbdc^{-1}a^{-1}b^{-1}ca\big)\big(c^{-1}a^{-1}b\big)\Big)\\
    &=d^{-1}cbdc^{-1}a^{-1}=\CR(x).
  \end{align*}
This fact holds for every trace, and shall be proven next.
\end{example}

\begin{lemma} \label{main core lemma} 
  Let $x \in \IRR(R)$ and $d=x\sqcap x^{-1}$.  Then
  $\NF_R(d^{-1}xd) = \CR(x)$.
\end{lemma}

\begin{proof}
  Let $d=x\sqcap x^{-1}$. Thus, there are traces $y,z$ such that $dy
  =x = z^{-1}d^{-1}$ and $\min(y) \cap \min(z) = \emptyset$. By Levi's
  Lemma it follows that there are traces $u,v_1,v_2,w$ such that
  $uv_1=d$, $v_2w=y$, $uv_2=z^{-1}$, $v_1w=d^{-1}$, and
  $v_1Iv_2$. Hence we have $v_1^{-1}\preceq d^{-1}$ and $v_1\preceq
  d^{-1}$ and since $x$ is $R$-irreducible, so is $d^{-1}$. We can apply
  Lemma~\ref{prefixinverseepsilon} to infer that $v_1=\varepsilon$.

  It follows that $u = d$, $w = d^{-1}$, and thus $x = dy = d v_2 w =
  d v_2 d^{-1}$.  Moreover, since $\min(v_2w) \cap \min (v_2^{-1}
  u^{-1}) = \min(y) \cap \min(z) = \emptyset$, we have
  $\min(v_2)\cap\min(v_2^{-1})=\emptyset$. Hence, $v_2$ is the core of
  $x$. Moreover since $x$ (and therefore $v_2$) is $R$-irreducible, we
  have $\NF_R(d^{-1}xd)=\NF_R(d^{-1}dv_2d^{-1}d)=v_2$.  \qed
\end{proof}
We now easily obtain:

\begin{corollary}\label{computecore}
  Fix an independence alphabet $(\Sigma^{\pm 1},I)$.  Then, the
  following problem can be solved in polynomial time:

  \medskip
  
  \noindent
  INPUT: An SLP $\dA$ \\
  OUTPUT: An SLP $\dB$ with $[\val(\dB)]_I = \CR([\val(\dA)]_I)$
\end{corollary}

\begin{proof}
  By Theorem~\ref{R-reduction} we can assume that $[\val(\dA)]_I \in \IRR(R)$.
  Then, using Theorem~\ref{theo-compressed-inf} we can
  compute in polynomial time an SLP $\dP$ with
  $[\val(\dP)]_I=[\val(\dA)]_I\sqcap[\val(\dA)^{-1}]_I$.  By
  Lemma~\ref{main core lemma} we have $\CR([\val(\dA)]_I)
  =\NF_R([\val(\dP)^{-1}\val(\dA)\val(\dP)]_I)$. Finally, by
  Theorem~\ref{R-reduction} we can compute in polynomial time an SLP
  $\dB$ such that
  $[\val(\dB)]_I=\NF_R([\val(\dP)^{-1}\val(\dA)\val(\dP)]_I)$.  \qed
\end{proof}

\section{A pattern matching algorithm for connected patterns}\label{sec:connected}

Our second tool for proving Theorem~\ref{decidesccp} is a pattern
matching algorithm for compressed traces.
For two traces $v$ and $w$ we say that $v$ is a factor of $w$ if there
is some trace $u$ such that $uv\preceq w$.  We consider the following
problem and show that it can be solved in polynomial time if the
independence alphabet $(\Sigma,I)$ satisfies certain conditions.

\medskip
\noindent 
INPUT: An independence alphabet $(\Sigma,I)$ and two SLPs $\dT$ and $\dP$
over $\Sigma$.\\
QUESTION: Is $[\val(\dP)]_I$ a factor of $[\val(\dT)]_I$?  
\medskip

\noindent
We write $\alp(\dT)$ and $\alp(\dP)$ for $\alp(\val(\dT))$ and
$\alp(\val(\dP))$, respectively.  We may assume that
$\Sigma=\alp(\dT)$ and that $\Sigma$ is connected. Otherwise we simply
solve several instances of the latter problem separately. Also, we
assume in the following that the SLPs $\dT=(V,\Sigma,S,P)$ and $\dP$
are in Chomsky normal form.  Let $\Gamma\subseteq \Sigma$. We denote
by $\pi_{\Gamma,}$ the homomorphism $\pi_{\Gamma}:\M(\Sigma,I)\to
\M(\Gamma,I\cap (\Gamma\times \Gamma))$ with $\pi_{\Gamma}(a)=a$ for
$a\in \Gamma$ and $\pi_{\Gamma}(a)=\epsilon$ for $a\in
\Sigma\setminus\Gamma$.  Let $V^\Gamma=\{X^\Gamma\mid  X\in V\}$ be a
disjoint copy of $V$. For each production $p\in P$ define a new
production $p^\Gamma$ as follows. If $p$ is of the form $X\to a$
$(a\in \Sigma)$, then let $p^\Gamma=(X^\Gamma \to a)$ in case $a\in
\Gamma$ and $p^\Gamma=(X^\Gamma\to \epsilon)$ otherwise. Moreover, if
$p\in P$ is of the form $X\to YZ$ $(X,Y,Z\in V)$ define
$p^\Gamma=(X^\Gamma\to Y^\Gamma Z^\Gamma)$.  We denote with $\dT^\Gamma$
the SLP $(V^\Gamma,\Gamma,S^\Gamma,P^\Gamma)$ where
$P^\Gamma=\{p^\Gamma\mid  p\in P\}$. 
Obviously, $\val(\dT^{\Gamma})=\pi_\Gamma(\val(\dT))$.

In order to develop a polynomial time algorithm for the problem stated
above we need a succinct representation for an occurrence of $\dP$ in
$\dT$. Since $[\val(\dP)]_I$ is a factor of $[\val(\dT)]_I$ iff there
is a prefix $u\preceq[\val(\dT)]_I$ such that $u[\val(\dP)]_I\preceq
[\val(\dT)]_I$, we will in fact compute prefixes with the latter
property and represent a prefix $u$ by its Parikh image $(|u|_a)_{a\in
  \Sigma}$.  Hence we say a sequence $O=(O_a)_{a\in \Sigma}\in
\N^{\Sigma}$ is an \emph{occurrence} of a trace $v$ in a trace $w$ iff
there is a prefix $u\preceq w$ such that $u v\preceq w$, and
$O=(|u|_a)_{a\in \Sigma}$.  For $\Gamma\subseteq \Sigma$ we write
$\pi_{\Gamma}(O)$ for the restriction
$(O_a)_{a\in\Gamma}$. Furthermore, we say that $O$ is an occurrence of
$\dP$ in $\dT$ if $O$ is an occurrence of $[\val(\dP)]_I$ in
$[\val(\dT)]_I$. Note that our definition of an occurrence of $\dP$ in
$\dT$ does not exactly correspond to the intuitive notion of an
occurrence as a convex subset of the dependence graph of
$[\val(\dT)]_I$. In fact, to a convex subset of the dependence graph
of $[\val(\dT)]_I$, which is isomorphic to the dependence graph
of $[\val(\dP)]_I$, there might correspond several
occurrences $O$, since for an $a\in \Sigma$ that is independent of
$\alp(\dP)$ we might have several possibilities for the value
$O_a$. However, if we restrict to letters that are dependent on
$\alp(\dP)$, then our definition of an occurrence coincides with the
intuitive notion.

Let $X$ be a nonterminal of $\dT$ with production $X\to YZ$ and let
$O$ be an occurrence of $[\val(\dP)]_I$ in $[\val(X)]_I$.  If there
are $a,b\in \alp(\dP)$ such that $O_a<|\val(Y)|_a$ and
$O_b+|\val(\dP)|_b>|\val(Y)|_b$, then we say that $O$ is an occurrence
of $\dP$ \emph{at the cut} of $X$.  We assume w.l.o.g. that
$|\val(\dP)|\geq 2$, otherwise the problem reduces simply to checking
whether there occurs a certain letter in $\val(\dT)$.  This
assumption  implies that $[\val(\dP)]_I$ is a factor of $[\val(\dT)]_I$ if
and only if there is a nonterminal $X$ of $\dT$ for which there is an occurrence
of $\dP$ at the cut of $X$.

\begin{example}\label{exapat1}
  We take the independence alphabet from Example \ref{Ex traces1}
  again. Let $X$ be a nonterminal with
  $\val(X)=acbc\;ad\;cbc\;acbc\;acbc\;acbc\;acb|c\;acbc\;acbc\;acbc\;acb\;dc$
  where '$|$' denotes the cut of $X$ and
  $\val(\dP)=acbc\;acbc\;acbc\;acbc\;acbc$. Then the occurrences of
  $\val(\dP)$ at the cut of $X$ are $(1,1,2,1)$,$(2,2,4,1)$,
  $(3,3,6,1)$ and $(4,4,8,1)$ where the positions in a tuple
  correspond to the letters in our alphabet in the order
  $a,b,c,d$. We will see later how to construct them.
\end{example}

\begin{lemma}[\cite{LiuWraZeg90}]
  \label{lem:liuwrazeg}
  Let $v$ and $w$ be traces over $\Sigma$. A sequence $(n_a)_{a\in
    \Sigma}\in \N^{\Sigma}$ is an occurrence of $v$ in $w$ if and only
  if $(n_a,n_b)$ is an occurrence of $\pi_{\{a,b\}}(v)$ in
  $\pi_{\{a,b\}}(w)$ for all $(a,b)\in D$.
\end{lemma}
\noindent
An \emph{arithmetic progression} is a subset of $\N^{\Sigma}$ of the form 
$$
\{ (i_a)_{a\in \Sigma}+k\cdot (d_a)_{a\in \Sigma}\mid  0\leq
k\leq \ell\}.
$$ 
This set can be represented by the triple
$((i_a)_{a\in\Sigma},(d_a)_{a\in\Sigma},\ell)$. The 
{\em descriptional size} $|((i_a)_{a\in\Sigma},(d_a)_{a\in\Sigma},\ell)|$ of the arithmetic
progression $((i_a)_{a\in\Sigma},(d_a)_{a\in\Sigma},\ell)$ is
$\log_2(\ell)+\sum_{a\in \Sigma}(\log_2(i_a)+\log_2(d_a))$.
In Example~\ref{exapat1}, the occurrences of $\val(\dP)$ at the cut of
$X$ form the arithmetic progression $\big((1,1,2,1),(1,1,2,0),3\big)$.

We will use the last lemma in order to compute the occurrences of
$\dP$ in $\dT$ in form of a family of arithmetic progressions. To this
aim, we follow a similar approach as Genest and Muscholl for message
sequence charts~\cite{GenMus08}. In particular Lemma~\ref{lem:sameNT}
below was inspired by \cite[Proposition~1]{GenMus08}.

Throughout the rest of this section we make the following assumption:
\begin{equation}
  \label{assumption}
  \text{$\alp(\dP)$ is
    connected  and $\{a,b\}\cap\alp(\dP)\ne
    \emptyset$ for all $(a,b) \in D$ with $a\ne b$.}
\end{equation}
Let $X$ be a nonterminal of $\dT$ and let $O$ be an occurrence of
$\dP$ at the cut of $X$.  Since the pattern is connected there must be
some $a,b\in \Sigma$ with $(a,b)\in D$ such that $\pi_{\{a,b\}}(O)$ is
at the cut of $X^{\{a,b\}}$.  We will therefore compute occurrences of
$\pi_{\{a,b\}}(\val(\dP))$ at the cut of $X^{\{a,b\}}$.  It is well
known that the occurrences of $\pi_{\{a,b\}}(\val(\dP))$ at the cut of
$X^{\{a,b\}}$ form an arithmetic progression
$((i_a,i_b),(d_a,d_b),\ell)$ and that $\pi_{\{a,b\}}(\val(\dP))$ is of
the form $u^nv$ for some $n\geq \ell$ and strings $u,v\in\{a,b\}^*$
with $v\preceq u$, $|u|_a=d_a$ and $|u|_b=d_b$.  Moreover, the
arithmetic progression $((i_a,i_b),(d_a,d_b),\ell)$ can be computed in
time $|\dT|^2|\dP|$ (see~\cite{Lif07}\footnote{In fact,
  in~\cite{Lif07} it was shown that the arithmetic progression
  $(i_a+i_b,d_a+d_b,\ell)$ can be computed in polynomial time. Observe
  that from this the arithmetic progression
  $((i_a,i_b),(d_a,d_b),\ell)$ can easily be computed.}).  Now suppose
we have computed the occurrences of $\pi_{\{a,b\}}(\val(\dP))$ at the
cut of $X^{\{a,b\}}$ in form of an arithmetic progression.  The
problem now is how to find (for the possibly exponentially many
occurrences in the arithmetic progression) matching occurrences of
projections onto all other pairs in $D$.

The following lemma states
that either there is a pair $(a,b)\in D$ such that the projection onto
$\{a,b\}$ is the first or the last element of an arithmetic
progression, or all projections lie at the cut of the same
nonterminal.

\begin{lemma}\label{lem:sameNT}
  Let $X$ be a nonterminal of $\dT$ and let $O$ be an occurrence of
  $\dP$ at the cut of $X$. Then either
 \begin{enumerate}[(i)]
 \item $\pi_{\{a,b\}}(O)$ is at the cut of $X^{\{a,b\}}$ for all
   $(a,b)\in D$ with $a\ne b$, or
 \item there are $a,b\in \alp(\dP)$ with $(a,b)\in D$ such that
   $\pi_{\{a,b\}}(O)$ is the first or last element of the arithmetic
   progression of occurrences of $\pi_{\{a,b\}}(\val(\dP))$ at the cut
   of $X^{\{a,b\}}$.
 \end{enumerate}
\end{lemma}

\begin{proof}
  Let $X\to YZ$ be a production of $\dT$. Clearly, by our general
  assumption \eqref{assumption} it suffices to show that either (ii)
  holds, or $O_a<|\val(Y)|_a<O_a+|\val(\dP)|_a$ for all $a\in
  \alp(\dP)$.  We show this assertion by induction on $|\alp(\dP)|$.
  If $\alp(\dP)$ is a singleton, then it is trivially true.

  Next, we consider the case $|\alp(\dP)|=2$. So let $\{a,b\}=
  \alp(\dP)$ and hence $(a,b)\in D$ by \eqref{assumption}.  
  Assume that (ii) does not hold.
  Consider the arithmetic progression $((i_a,i_b),(d_a,d_b),\ell)$ of
  occurrences of $\val(\dP)$ at the cut of $X^{\{a,b\}}$. Then
  $\val(\dP)$ is of the form $u^nv$ for some $n\geq \ell$ and strings
  $u,v\in\{a,b\}^*$ with $v\preceq u$, $|u|_a=d_a$ and $|u|_b=d_b$.
  We conclude that $d_a,d_b>0$ as otherwise $|\alp(\dP)|\leq 1$.
  Suppose for contradiction that $i_a+\ell d_a >|\val(Y)|_a$. Since no
  prefix $w$ of $\pi_{\{a,b\}}(\val(X))$ can satisfy $|w|_a<|\val(Y)|_a$
  and $|w|_b>|\val(Y)|_b$ we conclude $i_b+\ell
  d_b\geq|\val(Y)|_b$. But then the occurrence $(i_a+\ell d_a,i_b+\ell
  d_b)$ is not at the cut of $X^{\{a,b\}}$, which is a
  contradiction. Hence $i_a+\ell d_a\leq|\val(Y)|_a$ and by symmetry
  $i_b+\ell d_b\leq|\val(Y)|_b$.  Similarly, since $(i_a,i_b)$ is an
  occurrences of $\val(\dP)$ at the cut of $X^{\{a,b\}}$, we get
  $|\val(Y)|_a\leq i_a+|\val(\dP)|_a$ and $|\val(Y)|_b\leq
  i_b+|\val(\dP)|_b$.  As $\pi_{\{a,b\}}(O)$ is neither the first nor
  the last element of the arithmetic progression we have
  $O_a=i_a+kd_a$ and $O_b=i_b+kd_b$ for some $0<k<\ell$ and hence
  $O_a<|\val(Y)|_a<O_a+|\val(\dP)|_a$ and
  $O_b<|\val(Y)|_b<O_b+|\val(\dP)|_b$ as required.

  Now, suppose that $|\alp(\dP)|\geq 3$.  Since $O$ is an occurrence
  at the cut of $X$, there are $a,b\in \alp(\dP)$ such that
  $O_a<|\val(Y)|_a$ and $O_b+|\val(\dP)|_b>|\val(Y)|_b$.  We may
  assume that $(a,b)\in D$. Indeed, if $O_a+|\val(\dP)|_a>
  |\val(Y)|_a$ choose $a=b$. Otherwise, since $\alp(\dP)$ is connected
  there is a dependence path between $a$ and $b$. Since
  $O_a+|\val(\dP)|_a\leq |\val(Y)|_a$, there must be an edge
  $(a',b')\in D$ on this path such that 
  $a', b' \in \alp(\dP)$, $O_{a'}+|\val(\dP)|_{a'}\leq
  |\val(Y)|_{a'}$ (and hence $O_{a'} < |\val(Y)|_{a'}$), and
  $O_{b'}+|\val(\dP)|_{b'}>|\val(Y)|_{b'}$.

  Next consider a spanning tree of
  $(\alp(\dP),D\cap\alp(\dP)\times\alp(\dP))$ which contains the edge
  $(a,b)$ (in case $a\ne b$).  Let $c\notin\{a,b\}$ be a leaf of this
  spanning tree (it exists since $|\alp(\dP)| \geq 3$).  
  Obviously, $\Delta=\alp(\dP)\setminus\{c\}$ is
  connected and $\pi_\Delta(O)$ is at the cut of $X^\Delta$. Thus we
  can apply the induction hypothesis.  Assume again that (ii) does not
  hold. Applying the induction hypothesis to $\pi_\Delta(\val(\dP))$
  and $X^\Delta$ we get $O_a<|\val(Y)|_a<O_a+|\val(\dP)|_a$ for all
  $a\in \Delta$. In particular, $O_d<|\val(Y)|_d<O_d+|\val(\dP)|_d$
  for some $d\in \Delta$ with $(c,d)\in D$. Hence, $\pi_{\{d,c\}}(O)$
  is at the cut of $X^{\{d,c\}}$. Thus, applying the induction
  hypothesis also to $\pi_{\{d,c\}}(\val(\dP))$ and $X^{\{d,c\}}$ we
  get $O_c<|\val(Y)|_c<O_c+|\val(\dP)|_c$.\qed
\end{proof}
The last lemma motivates that we partition the set of occurrences into
two sets.  Let $O$ be an occurrence of $\dP$ in $\dT$ at the cut of
$X$. We call $O$ \emph{single} (for $X$) if there are
$a,b\in\alp(\dP)$ with $(a,b)\in D$ such that the projection
$\pi_{\{a,b\}}(O)$ is the first or the last element of the arithmetic
progression of occurrences of $\pi_{\{a,b\}}(\val(\dP))$ at the cut of
$X^{\{a,b\}}$. Otherwise, we call $O$ \emph{periodic} (for $X$).  By
Lemma~\ref{lem:sameNT}, if $O$ is periodic, then $\pi_{\{a,b\}}(O)$ is
an element of the arithmetic progression of occurrences of
$\val(\dP^{\{a,b\}})$ at the cut of $X^{\{a,b\}}$ for all $(a,b)\in D$
(but neither the first nor the last element, if $a,b\in\alp(\dP)$).
The next proposition shows that we can decide in polynomial time
whether there are single occurrences of $\dP$ in $\dT$.

\begin{proposition}\label{prop:computesingle}
  Given $a,b\in\alp(\dP)$ with $(a,b)\in D$, a nonterminal $X$ of
  $\dT$ and an occurrence $(O_a,O_b)$ of $\pi_{\{a,b\}}(\val(\dP))$ at
  the cut of $X^{\{a,b\}}$, we can decide in time
  $(|\dT|+|\dP|)^{O(1)}$ whether this occurrence is a projection of an
  occurrence of $\dP$ at the cut of $X$.
\end{proposition}

\begin{proof}
  Let $a_1,\ldots,a_n$ be an enumeration of $\Sigma$ such that
  $a=a_1$, $b=a_2$ and $D(a_i)\cap\{a_1,\ldots,a_{i-1}\}\ne \emptyset$
  for all $2\leq i\leq n$. Moreover, we require that the elements of
  $\alp(\dP)$ appear at the beginning of our enumeration, i.e., are
  the elements $a_1,\ldots,a_j$ for some $j\leq n$.  This can be
  assumed since $\Sigma$ and $\alp(\dP)$ are connected.  We iterate
  over $3\leq i\leq n$ and compute, if possible, an integer $O_{a_i}$
  such that $(O_{a_1},\ldots,O_{a_i})$ is an occurrence of
  $\pi_{\{a_1,\ldots,a_i\}}(\val(\dP))$ in
  $\pi_{\{a_1,\ldots,a_i\}}(\val(X))$.

  So let $i\geq 3$, $d=a_i$, and $\Delta=\{a_1,\ldots,a_{i-1}\}$.
  By our general assumption \eqref{assumption} we can choose some
  $c\in \Delta \cap\alp(\dP)$ such that $(c,d)\in D$.  Let us further
  assume that we have already constructed an occurrence
  $(O_{a_1},\ldots,O_{a_{i-1}})$ of $\pi_{\Delta}(\val(\dP))$ in
  $\pi_{\Delta}(\val(X))$.  First, we compute the number $k \geq 0$ such that
  $d^kc$ is a prefix of $\pi_{\{c,d\}}(\val(\dP))$.  Then, we compute
  $O_d$ such that there is a prefix $wd^kc$ of $\pi_{c,d}(\val(X))$
  for some $w\in \{c,d\}^*$ with $|w|_c=O_c$, $|w|_d=O_d$. If such a
  prefix does not exist, then there is no occurrence
  $(O_{a_1},\ldots,O_{a_{i-1}},O_{d})$ of
  $\pi_{\Delta \cup \{d\}}(\val(\dP))$ in
  $\pi_{\Delta\cup\{d\}}(\val(X))$. On the other hand, observe that
  if there is such an occurrence
  $(O_{a_1},\ldots,O_{a_{i-1}},O_{d})$, then $O_{d}=|w|_d$.
  Last, using~\cite{Lif07} we check in polynomial time for all $e\in
  D(d)\cap\Delta$ whether $(O_e,O_d)$ is an occurrence of
  $\pi_{\{d,e\}}(\val(\dP))$ in $\pi_{\{d,e\}}\val(X)$.  By
  Lemma~\ref{lem:liuwrazeg}, the latter holds if and only if 
  $(O_{a_1},\ldots,O_{a_{i-1}},O_d)$ is an occurrence of $\pi_{\Delta\cup\{d\}}(\val(\dP))$
  in $\pi_{\Delta\cup\{a_i\}}(\val(X))$.  \qed
\end{proof}
It remains to show that for every nonterminal $X$ of $\dT$ we can
compute the periodic occurrences. To this aim we define the
amalgamation of arithmetic progressions.  Let $\Gamma,\Gamma'\subseteq
\Sigma$ such that $\Gamma\cap\Gamma'\ne\emptyset$.  Consider two
arithmetic progressions
$$
p=((i_a)_{a\in \Gamma},(d_a)_{a\in \Gamma},\ell), \qquad
p'=((i'_a)_{a\in \Gamma'},(d'_a)_{a\in \Gamma'},\ell').
$$ 
The \emph{amalgamation} of $p$ and $p'$ is
$$
p\amalgam p'=\{v=(v_a)_{a\in \Gamma\cup\Gamma'}\mid  \pi_{\Gamma}(v)\in
p\text{ and } \pi_{\Gamma'}(v)\in p'\} .
$$

\begin{example}
  We continue Example \ref{exapat1} and show how to compute
  occurrences at the cut. First we consider the projections of $\dP$
  and $X$: 
  \begin{align*}
    \pi_{\{a,b\}}(\val(\dP)) &=(ab)^5   & \val(X^{\{a,b\}})&=(ab)^6|(ab)^4 \\
    \pi_{\{b,c\}}(\val(\dP))  &=(cbc)^5   & \val(X^{\{b,c\}})&=(cbc)^5cb|c(cbc)^4\\
    \pi_{\{c,d\}}(\val(\dP)) &= c^{10}  & \val(X^{\{c,d\}})&=c^2dc^9|c^8dc
  \end{align*}
  For the projections we find the arithmetic progressions
  $p_{ab},p_{bc},p_{cd}$ of
  occurrences at the cut:
  \begin{eqnarray*}
   \text{occurrences of } \pi_{\{a,b\}}(\val(\dP))\textrm{ at the cut of }X^{\{a,b\}}:&p_{ab}=\big((2,2),(1,1),3\big)\\
       \text{occurrences of }\pi_{\{b,c\}}(\val(\dP))\textrm{ at the cut of }X^{\{b,c\}}:&p_{bc}=\big((1,2),(1,2),4\big)\\
\text{occurrences of }\pi_{\{c,d\}}(\val(\dP)) \textrm{ at the cut of }X^{\{c,d\}}:&p_{cd}=\big((2,1),(1,0),7\big).
  \end{eqnarray*}
  Note that in $p_{ab}$ the first component corresponds to $a$ and the
  second to $b$ whereas in $p_{bc}$ the first component corresponds to
  $b$ and the second to $c$. We amalgamate the arithmetic progressions
   and obtain $p_{abc}=p_{ab}\amalgam
  p_{bc}=\big((2,2,4),(1,1,2),3\big)$.  If we again amalgamate we
  obtain $p_{abcd}=p_{abc}\amalgam p_{cd}=
  \big((2,2,4,1),(1,1,2,0),2\big)$.
  This way we found occurrences $(2,2,4,1)$, $(3,3,6,1)$ and
  $(4,4,8,1)$ of $\dP$ at the cut of $X$. Observe that there is a fourth
  occurrence $(1,1,2,1)$ that we did not find this way which is single.
\end{example}

\begin{lemma}\label{lem:intersectap}
  Let $\Gamma,\Gamma'\subseteq \Sigma$ with
  $\Gamma\cap\Gamma'\ne\emptyset$, and let $p=((i_a)_{a\in
    \Gamma},(d_a)_{a\in \Gamma},\ell)$ and $p'=((i'_a)_{a\in
    \Gamma'},(d'_a)_{a\in \Gamma'},\ell')$ be two arithmetic
  progressions. Then $p\amalgam p'$ is an arithmetic progression which
  can be computed in time $(|p|+|p'|)^{O(1)}$.
\end{lemma}
\begin{proof}
  We need to solve the system of linear equations
  \begin{equation}\label{eq:amalgam-systm}
    \left[~i_b+d_b\cdot x=i'_b+d'_b\cdot y~\right]_{b\in \Gamma\cap\Gamma'}
  \end{equation}
  for integers $x$ and $y$ under the constraint
  \begin{equation}\label{eq:amalgam-constr}
    0\leq x\leq \ell \text{ and } 0\leq y \leq \ell'.
  \end{equation}
  Let us fix an $a\in \Gamma\cap\Gamma'$. First we solve the single
  equation
  \begin{equation}\label{eq:amalgam-sing}
    i_a+d_a\cdot x=i'_a+d'_a\cdot y.
  \end{equation}
  for non-negative integers $x$ and $y$. The solutions are given by
  the least solution plus a mutliple of the least common multiple of
  $d_a$ and $d'_a$. We start by computing $g=\gcd(d_{a},d'_{a})$.  If
  $i_{a}\ne i'_{a} \mod g$, then there is no solution for
  equation~\eqref{eq:amalgam-sing} and hence $p\amalgam
  p'=\emptyset$. In this case we stop. Otherwise, we compute the least
  solution $s_{a}\geq \max(i_{a},i'_{a})$ of the simultaneous
  congruences
 \begin{align*}
   z&=i_{a}\mod d_{a},\\z&=i'_{a} \mod d'_{a}.
 \end{align*} 
 This can be accomplished with $(\log(d_a)+\log(d'_a))^2$ many bit
 operations; see e.g.~\cite{BacSha96}.  Let $k=(s_{a}-i_{a})/d_{a}
 \geq 0$
 and $k'=(s_a-i'_a)/d'_a \geq 0$. Now, the non-negative solutions of
 equation~\eqref{eq:amalgam-sing} are given by
 \begin{equation}\label{eq:amalgam-sol}
   (x,y)=(k+\frac{d'_a}{g}\cdot t,k'+\frac{d_a}{g}\cdot t) \text{ for all } t\geq
   0. 
 \end{equation}
 If $|\Gamma\cap\Gamma'|=1$ we adapt the range for $t$ such that the
 constraint~\eqref{eq:amalgam-constr} is satisfied and we are done.

 Otherwise, \eqref{eq:amalgam-systm} is a system of at least $2$ 
 linear equations in $2$ variables. Hence \eqref{eq:amalgam-systm} has
 at least $2$ (and then infinitely many) solutions iff any two equations
 are linearly dependent over $\mathbb{Q}$, i.e.  for all $b\in \Gamma\cap\Gamma'$ the
 following holds:
 \begin{equation}\label{eq:amalgam-condit}
   \exists k_b \in \mathbb{Q} : 
   d_a=k_b\cdot d_b,~~d'_b=k_b\cdot d'_a \text{ and }i'_a-i_a=k_b\cdot (i'_b-i_b)
 \end{equation}
 In this case all solutions of equation~\eqref{eq:amalgam-sing} are
 solutions of equation~\eqref{eq:amalgam-systm}. Thus we can test
 condition~\eqref{eq:amalgam-condit} for all $b\in \Gamma\cap\Gamma'$
 and in case it holds it only remains to adapt the range for $t$ such
 that the constraint~\eqref{eq:amalgam-constr} is satisfied.
 Otherwise there is at most one solution and we can fix $b\in
 \Gamma\cap\Gamma'$ such that \eqref{eq:amalgam-condit} does not
 hold. We plug the solution~\eqref{eq:amalgam-sol} into $i_b+d_b\cdot
 x=i'_b+d'_b\cdot y$ and obtain
 \begin{equation*}
   i_b+(k+ \frac{d'_a}{g}\cdot t)\cdot d_b =
   i'_b+(k'+\frac{d_a}{g}\cdot t)\cdot d'_b.
 \end{equation*}
 We can solve this for $t$ (if possible) and test whether this gives
 rise to a solution for \eqref{eq:amalgam-systm} under the
 constraint~\eqref{eq:amalgam-constr}.  \qed
\end{proof}

\begin{proposition}\label{prop:computeperiodic}
  Let $X$ be a nonterminal of $\dT$.  The periodic occurrences of
  $\dP$ at the cut of $X$ form an arithmetic progression which
  can be computed in time ${(|\dT|+|\dP|)^{O(1)}}$.
\end{proposition}

\begin{proof}
  As in the proof of Proposition~\ref{prop:computesingle} let
  $a_1,\ldots,a_n$ be an enumeration of $\Sigma$ such that
  $\{a_1,\ldots,a_{i-1}\}\cap D(a_i)\ne \emptyset$ for all $2\leq
  i\leq n$ and the elements of $\alp(\dP)$ appear at the beginning of
  the enumeration.  We iterate over $1\leq i\leq n$ and compute the
  arithmetic progressions of the periodic occurrences of
  $\pi_{\{a_1,\ldots,a_i\}}(\val(\dP))$ at the cut of
  $X^{\{a_1,\ldots,a_i\}}$. For $i=1$ this is easy.

  So let $i\geq 2$, let $a=a_i$ and let
  $\Delta=\{a_1,\ldots,a_{i-1}\}$.  Assume that the periodic
  occurrences of $\pi_{\Delta}(\val(\dP))$ at the cut of $X^{\Delta}$
  are given by the arithmetic progression
  $p=((i_c)_{c\in\Delta},(d_c)_{c\in\Delta},\ell)$. For all $b\in
  D(a)\cap \Delta$ let
  $$p^{\{a,b\}}=((i^{\{a,b\}}_a,i^{\{a,b\}}_{b}),(d^{\{a,b\}}_a,d^{\{a,b\}}_{b}),n^{\{a,b\}})$$
  be the occurrences of $\pi_{\{a,b\}}(\val(\dP))$ at the cut of
  $X^{\{a,b\}}$ (without the first and the last occurrence if
  $a,b\in\alp(\dP)$).  Recall that we assume that
  $\{c,d\}\cap\alp(\dP)\ne \emptyset$ for all $c,d\in \Sigma$ with
  $(c,d)\in D$ and $c \neq d$. Hence, by Lemma~\ref{lem:liuwrazeg},
  $O$ is a periodic occurrence 
  of $\pi_{\{a_1,\ldots,a_i\}}(\val(\dP))$ at
  the cut of $X^{\{a_1,\ldots,a_i\}}$ if and only if $\pi_\Delta(O)\in
  p$ and $(O_a,O_{b})\in p^{\{a,b\}}$ for all $b\in D(a)\cap \Delta$.
  Hence the periodic occurrences of
  $\pi_{\{a_1,\ldots,a_i\}}(\val(\dP))$ at the cut of
  $X^{\{a_1,\ldots,a_i\}}$ are given by
 $$ \bigotimes_{b\in D(a)\cap \Delta}p^{\{a,b\}} \otimes p.$$ 
 The result follows now from Lemma~\ref{lem:intersectap}.\qed
\end{proof}
Summarizing the last section we get the following theorem.

\begin{theorem}\label{thm:patmat}
  Given an independence alphabet $(\Sigma,I)$, and two SLPs $\dP$ and
  $\dT$ over $\Sigma$ such that $\alp(\dP)=\alp(\dT)$, we can decide
  in polynomial time whether $[\val(\dP)]_I$ is a factor of
  $[\val(\dT)]_I$.
\end{theorem}

\begin{proof}
  Note that our assumption  \eqref{assumption} is satsified
  if $\alp(\dP)=\alp(\dT)$.
  Recall that we may assume that $\alp(\dT)$ is connected and that
  $|\val(\dP)|\geq 2$.

  Let $X$ be a nonterminal of $\dT$.  Using \cite{Lif07} we compute
  for each pair $(a,b)\in D$ the arithmetic progression of occurrences
  of $\pi_{a,b}(\val(\dP))$ at the cut of $X^{\{a,b\}}$. By applying
  Proposition~\ref{prop:computesingle} to the first and to the last
  elements of each of these arithmetic progressions, we compute in
  polynomial time the single
  occurrences at the cut of $X$.  The periodic occurrences can be
  computed in polynomial time using Proposition~\ref{prop:computeperiodic}.  The result
  follows now since by definition $[\val(\dP)]_I$ is a factor of
  $[\val(\dT)]_I$ iff there is a nonterminal $X$ of $\dT$ such that
  there is either a single occurrence of $\dP$ at the cut of $X$ or a
  periodic occurrence of $\dP$ at the cut of $X$.  \qed
\end{proof}

\begin{remark}
  In the last section we actually proved the theorem above under
  weaker assumptions: We only need for each connected component
  $\Sigma_i$ of $\alp(\dT)$ that $\Sigma_i\cap \alp(\dP)$ is connected
  and that $\{a,b\}\cap\alp(\dP)\ne \emptyset$ for all $(a,b)\in
  D\cap (\Sigma_i\times\Sigma_i)$ with $a\ne b$.
\end{remark}

\section{Compressed conjugacy}
\label{sec:CC}

In this section we will prove Theorem~\ref{decidesccp}.  For this, we
will follow the approach from \cite{LiuWraZeg90,Wra89} for
non-compressed traces.  The following result allows us to transfer the
conjugacy problem to a problem on (compressed) traces:

\begin{theorem}[\cite{LiuWraZeg90,Wra89}] \label{conjugacymonoidgroup}
  Let $u,v\in \dM(\Sigma^{\pm 1},I)$. Then the following are
  equivalent:
  \begin{enumerate}[(1)]
  \item $u$ is conjugated to $v$ in $\dG(\Sigma,I)$.
  \item There exists $x \in \dM(\Sigma^{\pm 1},I)$ such that $x\,
    \CR(u) = \CR(v)\, x$ in $\dM(\Sigma^{\pm 1},I)$ (it is said that
    $\CR(u)$ and $\CR(v)$ are conjugated in $\dM(\Sigma^{\pm 1},I)$).
  \item $|\CR(u)|_a = |\CR(v)|_a$ for all $a \in \Sigma^{\pm 1}$ and
    there exists $k\leq |\Sigma^{\pm 1}|$ such that $\CR(u)$ is a
    factor of $\CR(v)^k$.
  \end{enumerate}
\end{theorem}
The equivalence of (1) and (2) can be found in \cite{Wra89}, the
equivalence of (2) and (3) is shown in \cite{LiuWraZeg90}.  We can now
infer Theorem~\ref{decidesccp}:

\medskip

\noindent
{\em Proof of Theorem~\ref{decidesccp}.}
Let $\dA$ and $\dB$ be two given SLPs over $\Sigma^{\pm 1}$.  We want
to check, whether $\val(\dA)$ and $\val(\dB)$ represent conjugated
elements of the graph group $\dG(\Sigma,I)$.  Using
Corollary~\ref{computecore}, we can compute in polynomial time SLPs
$\dC$ and $\dD$ with $[\val(\dC)]_I=\CR([\val(\dA)]_I)$ and
$[\val(\dD)]_I=\CR([\val(\dB)]_I)$.  By
Theorem~\ref{conjugacymonoidgroup}, it suffices to check the following
two conditions:
\begin{itemize}
\item $|\CR([\val(\dC)]_I)|_a = |\CR([\val(\dD)]_I)|_a$ for all $a \in
  \Sigma^{\pm 1}$
\item There exists $k\leq |\Sigma^{\pm 1}|$ such that
  $\CR([\val(\dC)]_I)$ is a factor of $\CR([\val(\dD)]_I)^k$.
\end{itemize}
The first condition can be easily checked in polynomial time, since
the number of occurrences of a symbol in a compressed strings can be
computed in polynomial time.  Moreover, the second condition can be
checked in polynomial time by Theorem~\ref{thm:patmat}, since (by the
first condition) we can assume that $\alp(\val(\dC)) =
\alp(\val(\dD))$.  \qed

\section{Open problems}

Though we have shown that some cases of the simultaneous compressed
conjugacy problem for graph groups (see
Section~\ref{sec:main-results}) can be decided in polynomial time, it
remains unclear whether this holds also for the general case.  It is
also unclear to the authors, whether the general compressed pattern
matching problem for traces, where we drop restriction 
\eqref{assumption}, can be decided in
polynomial time. Finally, it is not clear, whether 
Theorem~\ref{decidesccp}--\ref{thm:outer} also hold if the 
independence alphabet is part of the input.

\def\cprime{$'$}

\end{document}